\UseRawInputEncoding
\documentclass[11pt]{article}
\usepackage{amsfonts}
\usepackage{amsmath,amssymb,amsthm,latexsym}
\usepackage{amscd}
\usepackage[all]{xy}
\usepackage{hyperref}
\usepackage{mathrsfs}
\usepackage{fancyhdr}
\usepackage{indentfirst}
\usepackage{titlesec}
\usepackage{color}
\usepackage{subfigure}
\usepackage{tabularx}
\usepackage{multirow}
\usepackage{listings}
\usepackage{cite}
\usepackage{float}
\usepackage[all]{xy}
\usepackage{graphicx}
\usepackage{lineno}
\usepackage{marvosym}
\usepackage{geometry}
\usepackage{authblk}
\newcommand{\hlt}[1]{\textcolor{red}{#1}}
\newtheorem{theorem}{Theorem}[section]
\newtheorem{proposition}[theorem]{Proposition}
\newtheorem{corollary}[theorem]{Corollary}
\newtheorem{lemma}[theorem]{Lemma}
\newtheorem{problem}{Problem}
\newtheorem{newthe}[problem]{Theorem}
\theoremstyle{definition}
\newtheorem{defn}[theorem]{Definition}
\newtheorem{rem}[theorem]{Remark}
\newtheorem{example}[theorem]{Example}
\def\blue{\textcolor{blue}}

\def\rk{\mathrm{rank}}
\def\beq{\begin{equation}}
\def\eeq{\end{equation}}
\def\bspl{\begin{split}}
\def\espl{\end{split}}
\def\bgm{\begin{pmatrix}}
\def\edm{\end{pmatrix}}
\def\q{\theta}
\def\Q{\Theta}

\begin{document}
\title{Minimal Isometric Immersions of Flat $n$-tori into Spheres}
%Classification of Minimal isometric immersions by First Eigenfunctions for Conformally Flat  $3$-tori and $4$-tori in Spheres
\author{Ying L\"u\textsuperscript{1},\quad Peng Wang\textsuperscript{2},\quad Zhenxiao Xie %\thanks{Corresponding author}~~
\textsuperscript{\Letter~3}}

%\fi
\date{}
 \maketitle
\footnotetext[1]{School of Mathematical Sciences, Xiamen University, Xiamen, 361005, P. R. China. Email: {lueying@xmu.edu.cn}}
\footnotetext[2]{School of Mathematics and Statistics, Key Laboratory of Analytical Mathematics and Applications (Ministry of Education), FJKLAMA, Fujian Normal University, 350117 Fuzhou, P.R. China. Email: {pengwang@fjnu.edu.cn}}
\footnotetext[3]{School of Mathematical Sciences, Beihang University,  Beijing 100191, P. R. China. %
Email: {xiezhenxiao@buaa.edu.cn}}

\vspace{-0.25 cm}
\begin{abstract}
In 1985, Bryant established that a flat $2$-torus admits a minimal isometric immersion into some round sphere if and only if a certain rationality condition is satisfied. 
We show that when $n\geq 3$, the rationality criterion is no longer a necessary, but a sufficient condition for a flat $n$-torus to admit minimal isometric immersions into spheres. We also derive an upper bound for the algebraic irrationality degree of such immersions. When $n=3$, this bound is sharp and explicit embedded examples are provided respectively for each possible degree. Moreover, by constructing a family of non-homogeneous minimal flat $3$-tori, we show that minimal isometric immersions (embeddings) of flat $n$-tori are not necessarily homogeneous when $n \geq 3$. %to show that the minimal isometric immersion of flat $n$-tori is no longer necessarily homogeneous when $n\geq 3$.} 
%Moreover, we prove that every minimal isometric immersion of a flat $n$-torus can be deformed by a smooth homotopy of minimal isometric immersions into a homogeneous immersion in $\mathbb{S}^{p}\subset \mathbb{S}^m$ with $p<n^2+n$. %Moreover, we prove that every minimal flat $n$-torus in $S^m$ can be deformed minimally, isometrically and smoothly %by a smooth homotopy of minimal flat $n$-tori 
%into a homogeneous one in $\mathbb{S}^{p}\subset \mathbb{S}^m$ with $p<n^2+n$.  
In addition, we establish a deformation theorem that every flat $n$-torus admitting a minimal isometric spherical immersion %flat $n$-torus, if admitting a minimal flat immersion in some sphere, 
can  be isometrically,  minimally and homogeneously immersed into a sphere of dimension at most $n^2+n-1$. 
%}\fi
\end{abstract}


\indent{\bf MSC(2020):\hspace{2mm} 53C42, 53C40, 11H06}%, 58J50}
\iffalse
\vspace{-0.15 cm}
\setcounter{tocdepth}{1}
\tableofcontents
\fi
\section{Introduction}
The investigation of minimal isometric immersions of space forms into round spheres has been a subject of significant interest in differential geometry, owing to its deep connections with representation theory and spectral geometry. 

The $2$-dimensional 
case has been completely resolved through fundamental contributions by Calabi \cite{Calabi}, Kenmotsu \cite{Kenmotsu}, and Bryant \cite{Bryant}. The only minimal surfaces of constant positive Gaussian curvature 
in $\mathbb{S}^n$ are the round $2$-spheres and the {%(homogeneous) 
Veronese-Bor\'{u}vka} $2$-spheres. Surfaces of constant negative Gaussian curvature cannot be minimally and isometrically immersed into any sphere. While minimal isometric immersions of $\mathbb{R}^2$ into $\mathbb{S}^n$ can be explicitly parameterized, Bryant \cite[page 270]{Bryant} remarked that a flat $2$-torus $T^2 = \mathbb{R}^2/\Lambda_2$ admits a minimal isometric immersion into some sphere $\mathbb{S}^n$ if and only if some rational condition satisfies. In terms of the Gram matrix $Q$ of $\Lambda_2$, the rational condition can be stated as $Q$ has exclusively rational entries up to a dilation. For recent progress in this direction, we refer to the classification in \cite{Song} of minimal surfaces of constant negative curvature in a Hilbert sphere.  

For higher-dimensional space forms, only the case of positive sectional curvature has been investigated very well, primarily through the foundational work of do Carmo and Wallach \cite{doCarmo-Wallach}. Their work established a deep connection between minimal isometric immersions of spheres into spheres and representation theory, leading to what is now known as the do Carmo-Wallach theory. They showed there are in general many minimal isometric immersions of $\mathbb{S}^m$ into $\mathbb{S}^n$, and they can be parameterized by a compact convex body of finite dimension. These results were subsequently extended to general isotropy-irreducible Riemannian homogeneous spaces in \cite{Li,Wang-Ziller, TD}. The extension builds on Takahashi's foundational work \cite{Takahashi} establishing standard minimal isometric immersions   through the isotropy representation. There are also many works  devoted to investigating the moduli space of minimal isometric immersions of $\mathbb{S}^m$ into spheres and related problems, such as the determination of the exact dimension of this moduli space \cite{To1}, and the study of minimal target dimension \cite{Ma1, DZ, TZ, EW}. For further developments and additional references on related topics, we refer the reader to \cite{GT, E1, Toth, Tang} and the references therein. %These results were extended to general isotropy-irreducible Riemannian homogeneous spaces, i.e., if the isotropy group of a point on the tangent space, for which Takahashi \cite{Takahashi} showed standard minimal isometric immersions exist, see \cite{Li, Wang-Ziller, Deturck-Ziller} 

Since flat space forms are not isotropy-irreducible, the do Carmo-Wallach theory has limited applicability in this context. While it can be employed to study eigenmaps (i.e., harmonic maps with constant energy density) from flat $n$-tori into spheres \cite{Park-Urakawa, Park-Oh}, it does not extend in general to the investigation of those minimal isometric immersions. %of $n$-tori into spheres. 
A natural question in this field is to determine which flat \(n\)-tori, aside from the well-understood case of \(2\)-tori, admit minimal isometric immersions into spheres. %{A natural problem} in this field is that, except for the case of $2$-tori, it remains unknown {\em which flat $n$-tori admit minimal isometric immersions into spheres}. 
According to Moore's result \cite{Morre}, any flat $n$-torus admitting a minimal isometric immersion into $\mathbb{S}^{2n-1}$ must be either the Clifford $n$-torus or one of its coverings. For further results concerning isometric immersions of space forms into space forms, we refer to the survey articles \cite{Chen, Borisenko} and the book \cite{Dajczer}.  %A fundamental problem is that one do not know which flat $n$-tori admits a minimal isometric immersion into spheres except for the case of $2$-tori. 

Recently, the authors \cite{LWX} classified all minimal isometric immersions of flat $3,4$-tori into spheres by the first eigenfunctions (called $\lambda_1$-minimal isometric immersions for short). The $\lambda_1$-minimal  immersions of Riemannian manifolds into spheres are highly related to the concept of conformal volume in submanifold geometry and the first conformal spectrum in spectral geometry, see \cite{Li-Yau, Mon-Ros, Soufi-Ilias, Nadi, Tang-Xie-Yan, Tang-Yan, Narita, Karpukhin-Stern} and references therein. Our classification is based on establishing new variational characterizations for minimal flat $n$-tori, formulated through their underlying lattice structures (see Theorem~\ref{thm-vari} in Section~\ref{sec3}). After publishing our paper \cite{LWX}, we {realized} that $\lambda_1$-minimal isometric immersions of flat $n$-tori is  deeply connected to the semi-eutatic lattices  of rank $n$ in geometry of numbers \cite{Martinet}.  

%Our variational characterizations in \cite{LWX} lead immediately to the conclusion that a generic flat $n$-torus admits no minimal isometric immersion into any sphere, which was shown by Bryant \cite{Bryant} for the $2$-dimensional case. In this paper, we employ the variational characterizations to investigate the problem of determining which flat $n$-tori admit minimal isometric immersions into spheres. Our first result extends Bryant's criterion for flat 2-tori to arbitrary dimensions, providing a sufficient condition for minimal isometric immersions. We call a flat $n$-torus $\mathbb{R}^n/\Lambda_n$ rational, if the Gram matrix $Q$ in \eqref{eq-gram} of the lattice $\Lambda_n$ is rational (i.e., $Q\in GL(n,\mathbb{Q})$). 

{In this paper, we employ the aforementioned variational characterizations to investigate the problem of which flat $n$-tori admit minimal isometric immersions into spheres. An immediate consequence is that a generic flat $n$-torus admits no minimal isometric immersion into any sphere,  which was shown by Bryant \cite{Bryant} for the $2$-dimensional case. %Our first result 
We first extend Bryant's criterion for flat 2-tori to arbitrary dimensions, providing a sufficient condition for minimal isometric immersions. We call a flat $n$-torus $\mathbb{R}^n/\Lambda_n$ rational, if up to a dilation, the Gram matrix $Q$ in \eqref{eq-gram} of the lattice $\Lambda_n$ is rational (i.e., $Q\in GL(n,\mathbb{Q})$). }
\begin{newthe}\label{thm-1}
Up to a dilation, every flat rational $n$-torus admits a minimal isometric immersion into some sphere.  
\end{newthe}
The proof of Theorem \ref{thm-1} employs a classical result of Vorono\"i \cite{Voronoi} in geometry of numbers. %This approach not only establishes our general result but also provides an immediate verification of Bryant's assertion for the $2$-dimensional case \cite{Bryant}. 
Unlike the $2$-dimensional case, our explicit constructions demonstrate that the rationality condition is not necessary for the existence of minimal isometric immersions in %dimension $3$ and 
higher dimensions. Example~\ref{ex-rank5} and Examples~\ref{ex-4.4}$\sim$\ref{ex-rank4-2} demonstrate that minimal flat $3$-tori in spheres may be quartic, cubic, or quadratic irrational (see Section~\ref{sec-3tori} for precise definitions). %Example~\ref{ex-rank5} and  Example~\ref{ex-4.4} $\sim$ Example~\ref{ex-rank4-2} show that minimal flat $3$-tori in spheres may be quartic, cubic, quadratic irrational (see Section~\ref{sec-3tori} for precise definitions). 
In fact, these examples exhaust all possible degrees of field extensions over $\mathbb{Q}$ that can occur in the minimal isometric immersions of flat $3$-tori into spheres.   
{\begin{newthe}\label{thm-2}
   Let $T^n=\mathbb{R}^n/\Lambda_n$ be a flat $n$-torus, admitting a minimal isometric immersion into some sphere $\mathbb{S}^m$. Denote by $K$ %$K=\mathbb{Q}(q_{ij}/q_{11})$ 
   the field extension over $\mathbb{Q}$ generated by all ratios $q_{ij}/q_{11}$ of the Gram matrix $(q_{ij})$ of $\Lambda_n$. Then the extension degree satisfies 
 $$[K:\mathbb{Q}]\leq (n-1)^{\lfloor{\frac{(n-1)(n+2)}{4}}\rfloor},$$
 where $\lfloor \cdot \rfloor$ represents the integer part  of a real number. 
   %then its Gram matrix lies in $GL(n, \mathbb{Q}(w))\setminus GL(n, \mathbb{Q})$, and the extension degree of $\mathbb{Q}(w)/\mathbb{Q}$ is at most $4$. Moreover, if the minimal immersion of $T^3$ in $\mathbb{S}^m$ is full, and $[\mathbb{Q}(w):\mathbb{Q}]=3$ or $4$, then $m=7$.  
\end{newthe}}

For the $3$-dimensional case, the upper bound of extension degree established in Theorem \ref{thm-2} is equal to $4$, which is sharp. %Moreover, irrational minimal flat $3$-tori exhibit a certain rigidity: they must be homogeneous, and when $[K:\mathbb{Q}]$ is equal to $3$ and $4$, the minimal homothetic %isometric 
%immersion of such $3$-tori is unique. In particular, if such a $3$-torus admit a minimal immersion into spheres by the $k$-the eigenfunctions, then it cannot be minimally homotheticly immersed into spheres by any other $l$-th functions. 
Moreover, irrational minimal flat $3$-tori exhibit a certain rigidity: they must be homogeneous, and when $[K:\mathbb{Q}]$ equals $3$ or $4$, the minimal homothetic immersion of such a $3$-torus is unique up to congruence. In particular, if such a $3$-torus admits a minimal immersion into spheres via eigenfunctions corresponding to the $k$-th eigenvalue, then it cannot admit a minimal homothetic immersion via eigenfunctions belonging to any other eigenvalues. Note that this property fails to hold for minimal homothetic immersions of spheres, flat $2$-tori, or irrational flat $n$-tori into spheres.

In contrast to the homogeneous nature of minimal flat $2$-tori and minimal irrational flat $3$-tori, we construct a family of rational flat $3$-tori indexed by primes congruent to $1$ modulo $4$, each of which admits a $13$-parameter family of non-homogeneous minimal isometric immersions into spheres of dimension at most $23$. Any such minimal immersion that is full in $\mathbb{S}^{23}$ is in fact an embedding. We also construct a non-homogeneous minimal isometric immersion into $\mathbb{S}^7$ whose image has self-intersections. 
%\begin{newthe}\label{thm-2}
   %Let $T^3=\mathbb{R}^3/\Lambda_3$ be a flat $3$-torus. If it admits a minimal immersion into some sphere $\mathbb{S}^m$, then its Gram matrix lies in $GL(n, \mathbb{Q}(w))\setminus GL(n, \mathbb{Q})$, and the extension degree of $\mathbb{Q}(w)/\mathbb{Q}$ is at most $4$. Moreover, if the minimal immersion of $T^3$ in $\mathbb{S}^m$ is full, and $[\mathbb{Q}(w):\mathbb{Q}]=3$ or $4$, then $m=7$.  
%\end{newthe}
%Inspired by the works about target minimal dimension  

We further investigate the minimal target dimension for minimal isometric immersions of flat $n$-tori, and obtain the following deformation theorem.   
\begin{newthe}\label{thm-3}
    Let $x: T^n\triangleq \mathbb{R}^n/\Lambda_n\rightarrow  {\mathbb{S}^{2N-1}}$ be a minimal flat $n$-torus, {where $2N$ is the dimension of the eigenspace of $x$ with respect to the eigenvalue $n$}. Then $x$ can be deformed by a homotopy of minimal isometric immersions into a homogeneous immersion in $\mathbb{S}^{p}\subset \mathbb{S}^m$ with $p<n^2+n$.  
\end{newthe}
Setting $n={2}$ in Theorem \ref{thm-3}, we see that the minimal target dimension is $5$ for all flat $2$-tori admitting minimal isometric immersions into spheres, except for the Clifford $2$-torus and its covers. 

The paper is organized as follows. {In Section~\ref{sec3}, we begin by reviewing the spectral theory of flat $n$-tori and subsequently introduce our variational characterizations of their minimal isometric immersions. It is also shown in this section that the images of homogeneous minimal flat $n$-tori in spheres are embedded.  Section~\ref{sec-rational} focuses on minimal isometric immersions of rational flat $n$-tori, including the proof of Theorem~\ref{thm-1}. In Section~\ref{sec-3tori}, we first construct several examples of irrational minimal flat $3$-tori and establish Theorem~\ref{thm-2} in the case of $n=3$. {Then we address the rigidity and homogeneity of irrational minimal flat $3$-tori.} Section~\ref{sec-defor} is devoted to the existence of non-homogeneous rational minimal flat $3$-tori and to the study of their embeddedness. %to show that a flat, minimal $3$-torus is not necessarily homogeneous. %  Section~\ref{sec-defor}.  %Section~\ref{sec-defor} presents an example of a flat $3$-torus admitting non-homogeneous minimal isometric immersions. 
 Finally,  %
 we establish %the Deformation 
 Theorem~\ref{thm-3}, and {give a proof to Theorem~\ref{thm-2} for general $n$} %corse upper bound to the algebraic irrationality degree for minimal flat tori of general dimension 
 in Section~\ref{sec-deform}.
 } 

%The paper is organized as follows. In Section~\ref{sec3}, we first recall the basic spectral theory of flat $n$-tori, and then introduce the variational characterization of minimal isometric immersions of flat $n$-tori. Section~\ref{sec-rational} is devoted to investigating the minimal isometric immersion of rational flat $n$-tori, where Theorem~\ref{thm-1} is proved. We construct several examples of minimal flat irrational flat $3$-tori, and prove Theorem~\ref{thm-2} in Section~\ref{sec-3tori}. Finally, we present an example of flat $3$-torus which admits non-homogeneous minimal isometric immersions, and prove the deformation theorem in Section~\ref{sec-defor}.   
%we derive that,  except for the Clifford $2$-torus and its coverings,  $5$ is the minimal target dimension for flat $2$-tori admitting minimal isometric immersions into spheres.  
%The proof of this theorem also gives an argument to Bryant's claim on dimension $2$ in \cite{Bryant}.   {Section~\ref{sec-deform} is devoted to establishing Theorem~\ref{thm-2} for general $n$ and Theorem~\ref{thm-3}}. 
\section{Preliminaries}%{On minimal isometric immersions of %conformally
%flat tori}
\label{sec3}
In this section, we will recall the basic theory of flat $n$-tori and the setup developed by us for minimal isometric immersions of flat $n$-tori in \cite{LWX}, with particular emphasis on a variational characterization. %Then a sufficient condition for minimal flat tori in spheres to be homogeneous will be given.
%\vspace{3mm}
\subsection{Flat tori and lattices}

A flat torus $T^n$ of dimension $n$ can be described as
$$T^n=\mathbb{R}^n/{\Lambda_n},$$
where $\Lambda_n$ is a lattice of rank $n$ in $\mathbb{R}^n$. Set $L_n$ to be a generator matrix of $\Lambda_n$, which means $\Lambda_n$ can be generated by row vectors of $L_n$. The  Gram matrix of $\Lambda_n$ is then defined as 
\begin{equation}
    \label{eq-gram}
    Q:=L_n L_n^t.
\end{equation}

Two tori $T^n=\mathbb{R}^n/{\Lambda_n}$ and $\widetilde{T}^n=\mathbb{R}^n/{\widetilde{\Lambda}_n}$ are isometric if and only if $\Lambda_n$ and $\widetilde{\Lambda}_n$ are isometric, i.e., there exists an orthogonal matrix $O$ and a unimodular matrix $U\in SL(n,\mathbb{Z})$, such that $L_n=U\,\widetilde{L}_n\, O$, where $L_n$  (w.r.t. $\widetilde{L}_n$) is a generator matrix %$n\times n$ matrix whose row vectors are given by a basis
of lattice $\Lambda_n$ (w.r.t. $\widetilde{\Lambda}_n$).
It follows that the moduli space of flat $n$-tori  is
$$SL(n,\mathbb{Z})\setminus GL(n,\mathbb{R})\,/\,O(n).$$

The dual lattice of $\Lambda_n$ is defined as the lattice $\Lambda_n^{*}$ generated by row vectors of the matrix ${L}_n^*=(L_n^{-1})^t$.
%It is highly related to 
It is well known that the spectrum of $T^n=\mathbb{R}^n/{\Lambda_n}$ is given by %can be expressed %using the dual lattice $\Lambda_n^{*}$ as follows,   %, i.e., we have
$$\mathrm{Spec}(T^n)=\Big{\{}4\pi^2|\xi|^2\,\Big{|}\,\xi\in \Lambda_n^*\Big{\}},$$
and $e^{2\pi\langle\xi,u \rangle i}$ is an eigenfunction corresponding to the eigenvalue $4\pi^2|\xi|^2$, where
$u=(u_1,u_2,\cdots,u_n)$
is the %Cartesian
coordinates of $\mathbb{R}^n$, such that the flat metric on $\mathbb{R}^n$ ($T^n$) can be expressed as $du_1^2+d u_2^2+\cdots+d u_n^2$. 

For convenience, we introduce the following definition.
\vskip -0.6cm
\begin{defn}\label{def-ra}
Let $\mathbb{R}^n/\Lambda_n$ be a flat $n$-torus with $Q = (q_{ij})$ the Gram matrix of the lattice $\Lambda_n$. 
Consider the field extension $K$ of $\mathbb{Q}$ generated by the ratios $\left\{\frac{q_{ij}}{q_{11}} \mid 1 \leq i, j \leq n\right\}$.
\begin{itemize}
    \item[(1)] If $K = \mathbb{Q}$, we call $\mathbb{R}^n/\Lambda_n$ a \emph{rational flat $n$-torus}.
    \item[(2)] If the extension degree $[K:\mathbb{Q}] > 1$, we call $\mathbb{R}^n/\Lambda_n$ an \emph{irrational flat $n$-torus}.
\end{itemize}
\end{defn}
%\subsection{Minimal homogeneous flat tori in spheres}
\subsection{Algebraic characterizations of minimal isometric immersions of flat $n$-tori}
%The famous theorem \cite{Taka} of Takahashi says that a Riemannian manifold $(M^n,g)$ can be immersed minimally into a sphere  %$\mathbb{S}^m$ by $X$, if and only if, the coordinates of $X$ are eigenfunctions corresponding to the eigenvalue $n$.
Let $T^n=\mathbb{R}^n/{\Lambda_n}$ be a flat $n$-torus, which admits a minimal isometric immersion in some sphere. Then $n$ is an eigenvalue of $T^n$, whose eigenspace is assumed to be of dimension $2N$. We choose the coordinates $u=(u_1,u_2,\cdots,u_n)$ on $\mathbb{R}^n$ ($T^n$) such that the induced metric can be expressed as 
$$\frac{4\pi^2}{n}(du_1^2+du_2^2+\cdots+du_n^2).$$
It follows that there are exactly $N$ distinct lattice vectors  (up to $\pm1$) having the length $1$ %$\frac{\sqrt{n}}{2\pi}$ 
in the dual lattice $\Lambda_n^*$. We denote them by
$$\xi_1,\xi_2,\cdots,\xi_N.$$
%$X:T^n=\mathbb{R}^n/{\Lambda_n} \longrightarrow  \mathbb{S}^p$ be a linearly full minimal immersion. Assume
%$$\{\pm\xi_1,\pm\xi_2,\cdots\pm\xi_N\}$$
%enumerates all points in $\Lambda_n^*$ with length $n$.
By the theorem of Takahashi \cite{Takahashi}, a minimal isometric immersion of $T^n$ in spheres can be determined by some $2N\times 2N$ matrix $A$ as follows:
\beq\label{eq-X}
x=\bgm\Q_1&\Q_2&\cdots&\Q_N\edm A:T^n=\mathbb{R}^n/{\Lambda_n} \longrightarrow  \mathbb{S}^{2N-1},%~~~X=\bgm\Q_1&\Q_2&\cdots&\Q_N\edm A,
\eeq
where $\Q_r=\bgm\cos\q_r&\sin\q_r\edm,$ $\q_r=、2\pi\langle\xi_r,u\rangle$ for $~1\leq r\leq N$. %, and $A$ is a $2N\times 2N$ matrix. %Without loss of generality, we can assume $p=2N-1$ and
As in \cite{LWX}, we write 
\begin{equation}\label{eq-AAt}AA^t=\bgm A_{11}& A_{12}&\cdots& A_{1N}\\
 A_{21}& A_{22}&\cdots& A_{2N}\\
\vdots&\vdots&\ddots&\vdots\\
 A_{N1}& A_{N2}&\cdots& A_{NN}\edm,~~~ A_{rs}=\bgm A_{rs}^{11}& A_{rs}^{12}\\
 A_{rs}^{21}& A_{rs}^{22}\edm,
 \end{equation}
%we have the following conclusions. 
and define $\mathcal{E}=\{\xi_r\pm\xi_s,1\leq r<s\leq {N}\}$. 
\begin{defn}\label{de-etaset}
   Given $\eta\in \mathcal{E}$, we call the set of pairs  
 \[\{(\xi_{r_1},\xi_{s_1}), \cdots,(\xi_{r_p},\xi_{s_p}), (\xi_{r_{p+1}}, -\xi_{s_{p+1}}), \cdots, (\xi_{r_q}, -\xi_{s_q})\}\]
an $\eta$-set if $r_i<s_i$ and the relations
\begin{equation*}%\label{eq-etaset}
    %r_i<s_i,~~\text{and}~~
    \xi_{r_1}+\xi_{s_1}=\cdots=\xi_{r_p}+\xi_{s_p}=\xi_{r_{p+1}}-\xi_{s_{p+1}}=\cdots=\xi_{r_q}-\xi_{s_q}=\eta
\end{equation*}
   collectively exhaust all possible realizations of $\eta$. 
\end{defn}

    %Using $|\xi_{r_j}|=|\xi_{s_j}|$ we have
    %$$\langle \eta, \eta\rangle=|\xi_{r_j}|^2+|\xi_{s_j}|^2\pm 2\langle\xi_{r_j},\xi_{s_j}\rangle=2\langle \xi_{r_j},\eta\rangle>0,\quad 1\leq j\leq q.$$
    It is straightforward to verify that for a given $\eta$-set, the vectors  $\pm\xi_{r_1},\cdots,\pm\xi_{r_q}, \pm\xi_{s_1},\cdots,\pm\xi_{s_q}$ involved in it  are distinct with each other. 

%Denote by  $\xi_{r_{j}}\odot\xi_{s_{j}}$ the symmetric product of $\xi_{r_{j}}$ and $\xi_{s_j}$. 
It follows from the proof of Lemma~2.2 in \cite{LWX} that 
\begin{equation}\label{eq-Arr} A_{rr}=\bgm a_r&\\&a_r\edm,
\end{equation}
and the condition $|x|=1$ is equivalent to 
\begin{align}
    \sum_{\xi_r+\xi_s=\eta}(A_{rs}^{11}-A_{rs}^{22})+\sum_{\xi_r-\xi_s=\eta}(A_{rs}^{11}+A_{rs}^{22})=0, ~~~\forall \eta\in \mathcal{E}, \label{eq-eigen1}\\
    \sum_{\xi_r+\xi_s=\eta}(A_{rs}^{12}+A_{rs}^{21})+\sum_{\xi_r-\xi_s=\eta}(A_{rs}^{12}-A_{rs}^{21})=0, ~~~\forall \eta\in \mathcal{E}. \label{eq-eigen2}
\end{align}
Furthermore, the proof of Lemma~2.3 in \cite{LWX} implies that $x$ is an isometric immersion if and only if  
\begin{equation}\label{eq-euta}
    \sum_{r=1}^{N}a_r \xi_r \xi_r^t=\frac{I_n}{n}, 
\end{equation}
\vskip -0.5cm
\begin{align}
    \sum_{\xi_r+\xi_s=\eta}(A_{rs}^{11}-A_{rs}^{22})(\xi_r\xi_s^t+\xi_s\xi_r^t)+\sum_{\xi_r-\xi_s=\eta}(A_{rs}^{11}+A_{rs}^{22})(\xi_r\xi_s^t+\xi_s\xi_r^t)=O, ~~~\forall \eta\in \mathcal{E}, \label{eq-iso1}\\
    \sum_{\xi_r+\xi_s=\eta}(A_{rs}^{12}+A_{rs}^{21})(\xi_r\xi_s^t+\xi_s\xi_r^t)+\sum_{\xi_r-\xi_s=\eta}(A_{rs}^{12}-A_{rs}^{21})(\xi_r\xi_s^t+\xi_s\xi_r^t)=O, ~~~\forall \eta\in \mathcal{E}, \label{eq-iso2}
\end{align}
where $O$ denotes the $n\times n$ null matrix. 

Note that for a given $n$-torus $T^n=\mathbb{R}^n/\Lambda_n$,  %a couple of non-negative numbers $\{a_1, a_2, \cdots, a_N\}$ satisfying 
a solution of \eqref{eq-eigen1}$\sim$\eqref{eq-iso2} will provide a matrix of $2N\times 2N$. If it is positive semi-definite, then its square root gives us a minimal isometric immersion of $T^n$ in $\mathbb{S}^{2N-1}$ as in \eqref{eq-X}. In fact, solutions of \eqref{eq-eigen1}$\sim$\eqref{eq-iso2} parameterize the moduli space $\mathcal{M}(T^n)$ of all minimal isometric immersions of $T^n$ in $\mathbb{S}^{2N-1}$, which is {obviously} a convex set in $\mathrm{Sym}_{2N}$ if it is nonempty. Similar description for the minimal isometric immersions of isotropic irreducible Riemannian homogeneous space into spheres has been obtained (see do Carmo-Wallach \cite{doCarmo-Wallach}, Li \cite{Li}, Wang-Ziller \cite{Wang-Ziller}, Toth \cite{Toth} and references theirin), so has been for the eigenmaps of $n$-tori into spheres (see Park-Urakawa \cite{Park-Urakawa}). However, the difficulty here is that for general $n$-tori, the equations \eqref{eq-eigen1}$\sim$\eqref{eq-iso2} may have no solutions, as shown in the case of dimension $2$ by Bryant \cite{Bryant}. 

%For a general minimal and flat immersion of tori, one can 
{It is easy to see that \eqref{eq-eigen1},\eqref{eq-eigen2},\eqref{eq-iso1} and \eqref{eq-iso2} form a homogeneous linear system. So any non-trivial solution %(if existed) 
can be deformed continuously into the trivial solution. This indicates that the key equation is \eqref{eq-euta}, which characterizes the existence of homogeneous minimal isometric immersions of $T^n$ into $\mathbb{S}^{2N-1}$. This yields the following conclusion. 
\begin{proposition}\label{prop-homo}
A flat $n$-torus can be minimally and isometrically immersed into spheres if and only if it admits a homogeneous, minimal isometric  immersion into some sphere.   
\end{proposition}}
%The following lemma (see Lemma 2.3 in \cite{LWX}) provides a sufficient condition to
%\begin{rem}
    %According to Proposition \ref{prop-homo}, a strategy to find non-homogeneous $\lambda_1$-minimal flat tori might be the following: one can first classify the homogeneous $\lambda_1$-minimal flat tori, so all the possible sets $\{Y_i\}$ are obtained; then check all the $\eta$-sets to see if \eqref{eq-eigen1} and \eqref{eq-iso1} admit non-trivial solutions.
%\end{rem}
At the end of this subsection, we fix a notational convention: for a full minimal isometric immersion \(x: \mathbb{R}^n/\Lambda_n \rightarrow \mathbb{S}^m\) expressed in the form \eqref{eq-X}, the lattice vectors appearing in its coordinates may correspond only to a subset (not necessarily all) of the lattice vectors of length \(1\). These vectors will still be denoted by
\[
\xi_1, \xi_2, \dots, \xi_N
\]
throughout the subsequent discussion. 
\iffalse{At the end of this subsection, we make the following notation convention. 
\hlt{\begin{rem}
    Given a full minimal isometric immersion $x: \mathbb{R}^n/\Lambda_n \rightarrow \mathbb{S}^m$ formulated as in \eqref{eq-X},  the associated lattice vectors appearing in the coordinates of $x$ %(as in \eqref{eq-X}) 
    may correspond not to all lattice vectors of length $1$, but only to a subset of them. These vectors will still be denoted by
    $$\xi_1, \xi_2, \cdots, \xi_N$$ 
    in what follows.
\end{rem} 
}
}\fi
\subsection{A variational characterization} 
In \cite{LWX}, the solving of \eqref{eq-euta} is analyzed by us from a variational perspective. We provide a brief overview of it here.  Let $\{\eta_1,\eta_2,\cdots,\eta_n\}$ be a generator of the dual lattice $\Lambda_n^{*}$, and $Q$ be the Gram matrix of $\Lambda_n^*$ with respect to this generator. Then there exist integers $a_{j_k}$ such that
$$\xi_j=a_{j_1}\eta_1+a_{j_2}\eta_2+\cdots\cdots+a_{j_n}\eta_n,\quad 1\leq j\leq N.$$
%Consider a coordinates system with $\{\xi_1,\xi_2,\cdots,\xi_n\}$ as an orthonormal basis,
Set
$$Y_j=(a_{j_1},a_{j_2},\cdots\cdots,a_{j_n})^t,\quad Y=(Y_1,\cdots,Y_N). %\quad B_j= A_j^t A_j,
\quad 1\leq j\leq N.$$ 
We call $\{Q, Y\}$ a {\bf matrix data} of $x$. 
\iffalse{\begin{rem}
    Given a minimal isometric immersion $x: \mathbb{R}^n/\Lambda_n \rightarrow \mathbb{S}^m$,  the associated integer vectors may correspond not to the full set $Y$  but only to a subset  $\widetilde{Y}$. In such cases, we refer to $\{Q, \widetilde{Y}\}$ as a matrix data of $x$. 
\end{rem}}\fi

%$\{e_i=(e_{i1},e_{i2},e_{i3})^t\}_{i=1,2,3}$ be the generator of $\Lambda_3$, $\{e_i^*=(e^*_{i1},e^*_{i2},e^*_{i3})^t\}_{i=1,2,3}$  the corresponding generator of $\Lambda_3^*$. So $e_i^te^*_j=\delta_{ij}$. For any $u=(e_1,e_2,e_3)\mathbf{u}$, the $\theta_i$ defined in \eqref{eq-X} has an explicit expression
\iffalse
\[
\theta_i=2\pi\langle\xi_i,u\rangle=2\pi\langle (e_1^*,e_2^*,e_3^*)Y_i,(e_1,e_2,e_3)\mathbf{u}\rangle=2\pi\langle Y_i,\mathbf{u}\rangle.
\]
\fi

In terms of the matrix data, the condition $|\xi_j|=1$ is equivalent to
\beq\label{eq-AQA}
Y_j^tQY_j=1,%1\leq j\leq N,
\eeq
i.e.,
$Y_j$ lies on the hyper-ellipsoid $\mathcal{Q}$ determined by
$$\bgm
u_1\!&u_2\!&\cdots\!&u_n
\edm
Q\bgm
u_1\\u_2\\\vdots\\u_n
\edm=1.
$$
The equation \eqref{eq-euta} is equivalent to
\begin{equation}\label{eq:flat}
c_1^2 Y_1Y_1^t+c_2^2Y_2Y_2^t+\cdots\cdots+c_N^2Y_NY_N^t=\frac{1}{n}Q^{-1},
\end{equation}
i.e., $\frac{Q^{-1}}{n}$ lies in the convex hull spanned by $\{Y_jY_j^t\}_{j=1}^N$. 
We consider the space $\mathrm{Sym}_n$ of $n\times n$ symmetric matrices over $\mathbb{R}$, and equip it with the inner product:
$$\langle S_1, S_2\rangle=%\frac{1}{n}
\operatorname{tr}(S_1S_2),\quad S_1, S_2\in \mathrm{Sym}_n.$$
%For any given vector $v\in \mathbb{R}^n$, $\Pi_a(v):=\{M|\langle v^tv , M\rangle=vMv^t=a,a\in \mathbb{R}\}$ is an affine hyperplane dividing $S(n)$ into two half spaces:
%$$S^+_a(v)=\{M|vMv^t\geq a\}, ~~~S^-_a(v)=\{M|vMv^t\leq a\}.$$
Let $\Sigma$ (resp. $\Sigma_+$) be the set of semi-positive (resp. positive) definite matrices, which forms a close (resp. open) cone in $\mathrm{Sym}_n$. %Then it is well known that$\Sigma%\cup\{0\}=\cap_{v\in \mathbb{R}^n} S^+_0(v)$is a convex cone in $S(n)$ with the set of positive definite matrices as its interior, which is denoted  by $\Sigma_+$. By our definition, we have $Q^{-1}\in \Sigma_+$, and $ A_j^t A_j\in \Sigma, 1\leq j\leq N$.
Given a subset $X\subset\mathbb{Z}^n$, we denote by $C_X$ the convex hull spanned by $AA^t$ for all $A\in X$, and define  
$$W_X\triangleq\{M\,|\, M\in \Sigma_+, \langle M, AA^t\rangle=1\},$$ 
whose geometric meaning is the set of all hyper-ellipsoids passing through every point in $X$. %to represent all positive definite matrices orthogonal to. 
With the notations established, we present the variational characterization obtained in \cite{LWX}. %Now we state our variational characterization obtained in \cite{LWX}.   %and $V_X$  be the  affine subspace  $\cap_{A\in X}\Pi_{{1}}(A)\subset S(n)$.
%Moreover, we also consider the smooth linear submanifold $W_X= V_X\cap \Sigma_+$ (see Figure~\ref{fig:WX}), whose geometric meaning is the set of all hyper-ellipsoids passing through $X$. 
\begin{theorem}\label{thm-vari}
Let $x: T^n=\mathbb{R}^n/\Lambda_n\rightarrow \mathbb{S}^{m}$ be a minimal flat torus, and $\{Q, Y\}$ be the matrix data of $x$. Then $\frac{{Q}^{-1}}{n}$ is a maximum point of the determinant function restricted on $C_Y$, and $Q$ is a maximum point of the determinant function restricted on $W_Y$.  %an isometric homogeneous immersion, 
%where $N$ is the half dimension of the eigenspace of $T^n$ corresponding to $n$. If $x$ is minimal and linearly full, then $\frac{Q^{-1}}{n}$ lies in the interior of $C_Y$, and is a critical point of the  determinant function restricted on $C_Y$.

Conversely, given a finite set $X\subset \mathbb{Z}^n$ such that $C_X\cap \Sigma_+\neq \varnothing$, if $P\in \overset{\circ}{C_X}$ is a critical point of the determinant function restricted on $C_X$, then the torus $\mathbb{R}^n/\Lambda_n$ determined by $nP$ (as the Gram matrix of $\Lambda_n$) admits a minimal isometric immersion in some sphere. %$\mathbb{S}^{2N-1}$.
\end{theorem}
%and use $Y$ also representing the set of its row vectors if no confusion caused.  
\begin{rem}
    For any given integer set $X$, the critical point of the determinant function restricted on $C_X$ is determined by a system of algebraic equations with integer  coefficients, which implies flat $n$-tori that admit minimal isometric immersions into spheres have the property that the entries of their Gram matrices are algebraic {{numbers}}. %integers. 
    Thus, a generic flat $n$-torus cannot be minimally, isometrically immersed in  any spheres. When $n=2$, the aforementioned algebraic equations are all linear. This implies, as asserted by Bryant in \cite{Bryant}, that only rational flat $2$-tori could admit  minimal isometric immersions into spheres. \end{rem}
\begin{rem}\label{rk-em}
Let $\{\eta_1^*,\eta_2^*,\cdots,\eta_n^*\}$ be the dual frame of $\{\eta_1,\eta_2,\cdots,\eta_n\}$. It forms a generator of $\Lambda_n$. Let $\mathbf{u}=(\mathrm{u}_1, \mathrm{u}_2, \cdots, \mathrm{u}_n)$ be the coordinate of $u\in\mathbb{R}^n$ with respect to  $\{\eta_1^*,\eta_2^*,\cdots,\eta_n^*\}$. Then the function $\theta_i$ defined in \eqref{eq-X} has an explicit expression 
\begin{equation*}%\label{eq-YU}
\theta_i=2\pi\langle\xi_i,u\rangle=2\pi\langle (\eta_1,\cdots,\eta_n)Y_i,(\eta_1^*,\cdots,\eta_n^*)\mathbf{u}^t\rangle=2\pi\langle Y_i,\mathbf{u}^t\rangle. 
\end{equation*}
    In terms of the matrix data $\{Q, Y\}$, %it follows from \eqref{eq-YU} that 
    a homogeneous minimal immersion $x$ is embedded if and only if the image of the map 
    \begin{equation}\label{eq-em}
    \begin{split}
         \mathcal{Y}: \Omega\triangleq\{&(\mathrm{u}_1,\mathrm{u}_2,\cdots,\mathrm{u}_n)\;|\; 0\leq\mathrm{u}_i<1\} \longrightarrow  \mathbb{R}^N \\
         &~~~~~~~\mathbf{u}=(\mathrm{u}_1,\mathrm{u}_2,\cdots,\mathrm{u}_n)~~~~~\mapsto ~~\frac{1}{2\pi}(\theta_1,\theta_2, \cdots, \theta_N)=(\mathrm{u}_1,\mathrm{u}_2,\cdots,\mathrm{u}_n)Y 
    \end{split}
    \end{equation} %Y \cap \mathbb{Z}^N=\{0\}$$ 
    intersects $\mathbb{Z}^N$ solely at the origin, or equivalently, %The last condition is equivalent to say that 
    $\mathcal{Y}(\mathbf{u})$ is never a nonzero integer vector for any $\mathbf{u}\in \{(\mathrm{u}_1,\mathrm{u}_2,\cdots,\mathrm{u}_n)\;|\; |\mathrm{u}_i|<1\}$. When $N=n$, by the first Minkowski's  Convex Body Theorem \cite{Siegel}, this condition is equivalent to the determinant constraint 
 $\det(Y)=\pm 1$ (a similar characterization appeared in \cite{Park-Oh, Park-Urakawa}). %Similarly, 
    %It is easy to verify that w
    When $N> n$, if one of the $n\times n$ minor of $Y$ is equal to $\pm 1$, then $x$ is embedded. In fact, without loss of generality, we may assume the minors given by %first $n$ columns 
    $\{Y_1, Y_2, \cdots, Y_n\}$ is equal to $\pm1$. Then the first Minkowski's Convex Body Theorem \cite{Siegel} implies that the origin is the unique point in $\Omega$ such that $\frac{\theta_1}{2\pi}, \frac{\theta_2}{2\pi}, \cdots,\frac{\theta_n}{2\pi}$ are integers; hence the origin is the unique integer point in $\mathcal{Y}(\Omega)$. %Even when no $n\times n$ minor of $Y$ equals $\pm 1$, if the restriction of $\mathcal{Y}$ to the subcube $\{(\mathrm{u}_1,\mathrm{u}_2,\dots,\mathrm{u}_n) \mid 0 \le \mathrm{u}_i < 1\}$ contains no nonzero integer vectors, the immersion $x$ is still embedded. 
\end{rem}
\begin{proposition}\label{prop-em}
    Any homogeneous minimal isometric immersion of a flat $n$-torus $T^n$ into a sphere can be realized as a minimal isometric embedding of a flat $n$-torus $\widetilde{T}^n$ which is covered by $T^n$. %is a covering of $\widetilde{T}^n$. 
\end{proposition}
\begin{proof}
\iffalse{
If the immersion of $T^n$ is not an embedding, then there is a nonzero vector $\mathbf{u}=(u_1,u_2,\cdots,u_n)$ such that $(u_1,u_2,\cdots,u_n)Y$ is an integral vector for some $|u_i|<1$. Suppose $u_1=p/q\not=0$ for some integers $p<q$ and $(p,q)=1$. Then there are integers $m_1$ and $m_2$ such that $m_1p+m_2q=1$. Since
\[
\mathbf{u}'=(mu_1-[mu_1],mu_2-[mu_2],\cdots,mu_n-[mu_n])=(1/q,*,\cdots,*)
\]
also satisfies $\mathbf{u}'Y$ being integral, we may assume $u_1=1/q$ in the first place.
}\fi
Let $x: T^n=\mathbb{R}^n/\Lambda_n\rightarrow \mathbb{S}^N$ be a homogeneous minimal flat $n$-torus with $\{Q, Y\}$ as its matrix data. 

Note that the map defined in \eqref{eq-em} can be extended to a linear map, still denoted by $\mathcal{Y}$, from $\mathbb{R}^n$ to $\mathbb{R}^N$. Since $Y$ is an integer matrix with $\operatorname{rank}(Y) = n$, the map $\mathcal{Y}: \mathbb{R}^n \longrightarrow \mathbb{R}^N$ 
%\[
%\mathcal{Y}|_{\mathbb{Z}^n} : \mathbb{Z}^n \longrightarrow \mathbb{Z}^N
%\]
is an injective group homomorphism with respect to the addition group structure of $\mathbb{R}^n$ and $\mathbb{R}^N$, and its image $\mathcal{Y}(\mathbb{Z}^n)$ forms a sublattice of $\mathbb{Z}^N$ of rank $n$. Consequently, the preimage 
\[
L_n \triangleq \mathcal{Y}^{-1}(\mathbb{Z}^N)
\] 
is itself a lattice of rank $n$, with $\mathbb{Z}^n$ as its sublattice. 

If $\Omega \cap (L_n \setminus \{\mathbf{0}\}) = \emptyset$, then by Remark~\ref{rk-em} the map $x \colon T^n \to \mathbb{S}^N$ is already a minimal embedding, completing the proof. Otherwise, there  exists a generator $\{\mathbf{v}_1, \mathbf{v}_2, \dots, \mathbf{v}_n\}$ of $L_n$ such that each $\mathbf{v}_i$ lies in $\overline{\Omega}$ and at least one belongs to $\Omega$. Consider the generator matrix $V$ of $L_n$ formed by $\{\mathbf{v}_1, \mathbf{v}_2, \dots, \mathbf{v}_n\}$ as row vectors. Since $\mathbb{Z}^n$ is a sublattice of $L_n$, the dual lattice $L_n^*$ is a sublattice of $(\mathbb{Z}^n)^*= \mathbb{Z}^n$. Therefore, $\mathrm{V}^{-1}$ is an integer matrix.

Set 
$$\widetilde{Y}\triangleq\mathrm{V}Y,\quad\quad\quad \widetilde{Q}\triangleq(\mathrm{V}^{-1})^{t}\,Q\,\mathrm{V}^{-1}.$$
%Since $\{\mathbf{v}_1, \mathbf{v}_2, \dots, \mathbf{v}_n\}\subset \Omega_1$, 
By definition, every column of $\widetilde{Y}$ is an integer vector. 
Denote by $\widetilde{\Lambda}_n$ the lattice with $\widetilde{Q}^{-1}$ as its Gram matrix, and let $\widetilde{T}^n \triangleq \mathbb{R}^n/\widetilde{\Lambda}_n$ be the associated flat $n$-torus. It follows that %Next we prove that 
$T^n$ is a Riemannian covering of $\widetilde{T}^n$. %This reduces to showing that $V^{-1}$ is an integer matrix, which follows directly from the claim below.  

Set $\widetilde{\mathbf{u}}=(\widetilde{\mathrm{u}}_1, \widetilde{\mathrm{u}}_2,\cdots, \widetilde{\mathrm{u}}_n)\triangleq\mathbf{u}\mathrm{V}^{-1}$. %Moreover, 
By \eqref{eq-em}, we have 
$$\theta_i=2\pi\langle Y_i, \mathbf{u}^t\rangle=2\pi\langle \widetilde{Y_i}, \widetilde{\mathbf{u}}^t\rangle,$$
which implies that the map $x$ induces a minimal isometric immersion of $\widetilde{T}^n$ into $\mathbb{S}^N$. Note that the image of the map  
\begin{equation*}%\label{eq-em}
    \begin{split}
         \widetilde{\mathcal{Y}}: \widetilde{\Omega}\triangleq\{&(\widetilde{\mathrm{u}}_1, \widetilde{\mathrm{u}}_2,\cdots, \widetilde{\mathrm{u}}_n)\;|\; 0\leq\widetilde{\mathrm{u}}_i<1\} \longrightarrow  \mathbb{R}^N \\
         &~~~~~~~\widetilde{\mathbf{u}}=(\widetilde{\mathrm{u}}_1, \widetilde{\mathrm{u}}_2,\cdots, \widetilde{\mathrm{u}}_n)~~~~~\mapsto ~~\frac{1}{2\pi}(\theta_1,\theta_2, \cdots, \theta_N)=(\widetilde{\mathrm{u}}_1, \widetilde{\mathrm{u}}_2,\cdots, \widetilde{\mathrm{u}}_n)\widetilde{Y} 
    \end{split}
    \end{equation*} 
    is exactly that of $\mathcal{Y}$ restricted on the parallel polyhedron spanned by $\{\mathbf{v}_1, \mathbf{v}_2, \dots, \mathbf{v}_n\}$. Then, the generator property of %our choice of 
    $\{\mathbf{v}_1, \mathbf{v}_2, \dots, \mathbf{v}_n\}$ ensures that $\widetilde{\mathcal{Y}}(\widetilde{\Omega})$ contains no  integer vector but zero. Consequently, the minimal isometric immersion $x: \widetilde{T}^n\rightarrow \mathbb{S}^N$ is an embedding by Remark~\ref{rk-em}. 
\end{proof}
Because every minimal flat $2$-torus is homogeneous, the following holds. 
\begin{corollary}\label{coro-em}
    Every minimal isometric immersion of a flat $2$-torus $T^2$ into a sphere factors through a minimal isometric embedding of a flat $2$-torus $\widetilde{T}^2$ that is finitely covered by $T^2$.
\end{corollary}
\section{Minimal isometric immersions of rational flat $n$-tori into spheres}\label{sec-rational}%$\mathbb{S}^N$}
%\subsection{Rational flat tori}
Let $\mathbb{R}^n/\Lambda_n$ be a rational flat $n$-torus (see Definition~\ref{def-ra}). Up to a dilation, we may assume that the Gram matrix of $\Lambda_n$ lies in $GL(n,\mathbb{Q})$, i.e., all its entries are rational numbers. 
Note that this condition is equivalent to requiring that the Gram matrix of the dual lattice $\Lambda_n^*$ also belongs to $GL(n,\mathbb{Q})$.

Given a positive symmetric matrix $Q\in GL(n,\mathbb{R})$, we use the notation $\mathcal{Q}$ to denote the hyper-ellipsoid determined by $u^t\,Q\,u=1$. 
%\iffalse
\begin{lemma}\label{lem-rational}
    Let $Q\in GL(n,\mathbb{Q})$ be a positive definite symmetric matrix.  Then up to a dilation, the rational points on the hyper-ellipsoid $\mathcal{Q}$ are dense. 
\end{lemma}
\begin{proof}
    First, note that there exists a point $u_0$ on the hyper-ellipsoid $\mathcal{Q}$ and a nonzero real number $\lambda$ such that $\frac{u_0}{\lambda}$ is a rational point. In fact, one can consider a line passing through the origin with a direction determined by a rational vector. The intersection of this line with $\mathcal{Q}$ is nonempty. 

    Let $\Pi$ be a coordinate hyperplane that %does 
    not containing the point $u_0$. For any given point $u'\in \Pi$, the line passing through $u_0$ and $u'$ intersects the hyper-ellipsoid $\mathcal{Q}$ at the point %$u$ with %, which is equal to 
    $$u=u_0-2\frac{u_0 Q (u'-u_0)^t}{(u'-u_0) Q (u'-u_0)^t}\,(u'-u_0).$$
    Then the conclusion follows from the fact that the points set $\{u'\in \Pi \;|\; \frac{u'}{\lambda} \text{is~rational}\}$ is dense in $\Pi$. 
\end{proof}
%\fi
\iffalse 
\begin{corollary}\label{cor-rational}
     Let $Q\in GL(n,\mathbb{Q})$ be a positive symmetric matrix. Then up to a dilation, there exists $N\triangleq\frac{n(n+1)}{2}$ rational points $\{Y_1, \cdots, Y_N\}$ on the hyper-ellipsoid $\mathcal{Q}$ %determined by $x\,Q\,x^t=1$ 
     such that 
     $\{Y_1 Y_1^t, Y_2 Y_2^t, \cdots Y_N Y_N^t \}$ 
     spans the whole space $\mathbb{R}^N$. 
\end{corollary}
\begin{proof}
    If we do not require the rationality, such points are easy to be found, such as the intersection of rays directed by $\{e_j+e_k\;|\; 1\leq j\leq k\leq n\}$ with the hyper-ellipsoid $\mathcal{Q}$. For the rationality, we only need to perturb such points using   
    %$$\{e_1, e_2, \cdots, e_n, e_j+e_k\}$$
    the dense property of rational points on $\mathcal{Q}$. 
    %ince there always exists $N\triangleq\frac{n(n+1)}{2}$ points $\{\}$ on $\mathcal{Q}$ such that there
\end{proof}
\fi

\begin{lemma}\label{lem-rational1}
     Let $Q$ %\in GL(n,\mathbb{R})$ 
     be a positive definite symmetric matrix. Then there exist $N=\frac{n(n+1)}{2}$ points $\{Y_1, \cdots, Y_N\}$ on the hyper-ellipsoid $\mathcal{Q}$ %determined by $x\,Q\,x^t=1$ 
     such that 
     $\{Y_1 Y_1^t, Y_2 Y_2^t, \cdots Y_N Y_N^t \}$ 
     spans the whole space $\mathrm{Sym}_n$. %$\mathbb{R}^N$. 
\end{lemma}
\begin{proof}
    %If we do not require the rationality, 
    Such points can be easily found by considering the intersection of rays directed by $\{e_j+e_k\;|\; 1\leq j\leq k\leq n\}$ with the hyper-ellipsoid $\mathcal{Q}$. %For the rationality, we only need to perturb such points using the dense property of rational points on $\mathcal{Q}$.   
    %$$\{e_1, e_2, \cdots, e_n, e_j+e_k\}$$
    %ince there always exists $N\triangleq\frac{n(n+1)}{2}$ points $\{\}$ on $\mathcal{Q}$ such that there
\end{proof}

In geometry of number, for a given matrix (quadratic form) $Q\in \Sigma_+$ %\subset\mathrm{Sym}_n$, 
the Vorono\"i domain $\mathcal{V}(Q)$ is defined as  
\[\mathcal{V}(Q):=\hbox{ the cone spanned by } \{uu^t\;|\; u\in \mathrm{Min}(Q)\},\] where $\mathrm{Min}(Q)$ denotes the set of shortest vectors of the lattice determined by $Q$. $Q$ is called a perfect form if the rank of $\{uu^t\;|\; u\in \mathrm{Min}(Q)\}$ equals $\frac{n(n+1)}{2}$. In \cite{Voronoi}, Vorono\"i established the finiteness of $n$-dimensional perfect forms modulo $GL(n,\mathbb{Z})$-equivalence and derived the following result. 
%proved that the number of $n$-dimensional perfect forms is finite up to the $GL(n,\mathbb{Z})$ equivalence, and obtained the following result.     
%$$\mathcal{V}(Q)\triangleq\{\}$$
%is called perfect, if  %positive symmetric matrix $lattices (quadratic forms) is needed. In 
\begin{lemma} \label{lem-Voronoi}%{$\mathrm{(Vorono\"i)}$}
    The cone $\Sigma_+\subset \mathrm{Sym}_n$ of positive definite matrices can be covered by the Vorono\"i domains of the $n$-dimensional perfect forms, i.e., 
    $$\Sigma_+\subset \bigcup_{\begin{array}{c}Q_p\;\rm{is~perfect}\\ \lambda_1(Q_p)=1\end{array}}\mathcal{V}(Q_p),$$
    where $\lambda_1(Q_p)$ denotes the shortest length of vectors in the lattice determined by $Q_p$. 
\end{lemma}

\begin{theorem}\label{thm-rational}
    Up to a dilation, every rational flat $n$-torus admits a minimal isometric immersion into some spheres.  
\end{theorem}
\begin{proof}
Let $\mathbb{R}^n/\Lambda_n^*$ be a rational flat $n$-torus, and $Q\in GL(n,\mathbb{Q})$ be a Gram matrix of $\Lambda_n$ with respect to some generator. Assume that $Y_1, \cdots, Y_N$ are points on the hyper-ellipsoid $\mathcal{Q}$ satisfying the property of Lemma~\ref{lem-rational1}. 
%We choose $Y_1, \cdots, Y_N$ on the hyper-ellipsoid $\mathcal{Q}$ so that 

In the space $\mathrm{Sym}_n$, we consider the affine hyperplane $\Pi_Q$ defined by %denote by $\Pi_Q$ the affine hyperplane defined as 
$$\Pi_Q\triangleq\{P\;|\; \langle P, Q\rangle =1\}.$$
%which forms an affine hyperplane. 
Note that all of $Y_1 Y_1^t, Y_2 Y_2^t, \cdots, Y_N Y_N^t$ belong to $\Pi_Q$, and %. Moreover, it follows 
from Lemma~\ref{lem-rational1}, we have  %that 
$$\Pi_Q=\{\lambda_1 Y_1 Y_1^t+\lambda_2 Y_2 Y_2^t+ \cdots+\lambda_N Y_N Y_N^t\;|\; \lambda_1, \lambda_2, \cdots, \lambda_N\in \mathbb{R}, \lambda_1+\lambda_2+ \cdots +\lambda_N=1\}.$$
%Since $\mathrm{rank}\{Y_1 Y_1^t, Y_2 Y_2^t, \cdots Y_N Y_N^t\}=N$,  

Obviously, $\frac{Q^{-1}}{n}$ is contained in $\Pi_Q$, so $\Pi_Q\cap \Sigma_{+}\not=\emptyset$. Since  $\ln\circ \det$ on $\Sigma_+$ is strictly concave,  the maximum point exists and is unique on $\Pi_Q\cap \Sigma_{+}$ for the determinant function $\det$. We denote it by %which is denoted by 
$Q_1$. It follows from the proof of Theorem 3.6 in \cite{LWX} that $$\langle \frac{Q_1^{-1}}{n},Y_i Y_i^t\rangle =1,~~~1\leq i\leq N.$$ Using Lemma~\ref{lem-rational1} again, we obtain that $Q_1=\frac{Q^{-1}}{n}$. 

By Lemma~\ref{lem-Voronoi}, $Q^{-1}/n$ is contained in some Vorono\"i domain $\mathcal{V}(Q_p)$ with vertices $u_iu_i^t$ ($1\leq i\leq N)$. However, it's location could be on the boundary of $\mathcal{V}(Q_p)$. To proceed, we need to find an open cone containing $Q^{-1}/n$.
Because $Q^{-1}/n$ is positive-definite, there exists $\epsilon>0$ so small that $Q^*=Q^{-1}/n-\epsilon\sum_{i=1}^N u_iu_i^t$ is still positive-definite. Analogously, $Q^*$ is also contained in some Vorono\"i domain. We may assume that
\[
Q^*=\sum_{i=1}^k c_i v_i{v_i}^t,\quad c_i>0,~~k\leq N,
\]
which leads to
\[
\frac{Q^{-1}}{n}=\sum_{i=1}^k c_i v_i{v_i}^t+\epsilon\sum_{i=1}^N u_iu_i^t,
\]
i.e., $\frac{Q^{-1}}{n}$ is contained inside the polytope with vertices $v_i{v_i}^t$ and $u_ju_j^t$. Since $v_i^tQv_i$ and $u_j^tQu_j$ are all positive, we may assume that all of them equal $1$ after performing the necessary rescaling on $v_i$ and $u_j$, which implies that all these  points lie on the hyper-ellipsoid $\mathcal{Q}$. %By rescaling we assume that all the $v_i{v_i}^t$ and $u_ju_j^t$ belonging to $\Pi_Q$, which means all the points $v_i$ and $u_j$ lie on the hyper-ellipsoid $\mathcal{Q}$. 
From Lemma~\ref{lem-rational} and the rationality assumption of $Q$, after applying a certain dilation if necessary, there are finite rational points $R_1, R_2, \cdots, R_{N+k}$ on $\mathcal{Q}$ so close to $v_i$ and $u_j$ that $\frac{Q^{-1}}{n}$ can be expressed as a convex combination of 
$$\{R_1 R_1^t, R_2 R_2^t,\cdots,R_{N+k} R_{N+k}^t\}.$$ Multiplying a suitable integer $\mu$ to $R_1,\cdots, R_{N+k}$, we  obtain a set 
$$Y=\{\mu R_1, \mu R_2, \cdots, \mu {R}_{N+k}\}\subset \mathbb{Z}^n.$$ %of integer points. 
Then $\{Q/\mu^2,Y\}$ %\frac{Q}{\epsilon^2}\}$ 
forms a matrix data satisfying the assumption of Theorem~\ref{thm-vari}; hence it provides a minimal isometric  immersion of $\mathbb{R}^n/\mu\Lambda_n^*$.  %we  $\{X_1, \codts, X_m\}$, we can obtain a matrix data  

 \end{proof}
In the above proof, % of the above theorem, 
the density property implies there are infinitely many choices of rational points $\{R_1,\cdots,R_{N+k}\}$, and so there exist  infinitely many choices of $\mu$ and $Y$ in the construction of matrix data. As a consequence, we derive the following proposition, which extends another result previously asserted by Bryant for the $2$-torus in \cite{Bryant}. 
\begin{proposition}\label{eq-infinite}
    For every rational flat $n$-torus $T^n$, there exist infinitely many positive integers $k$ so that $T^n$ can be minimally immersed into spheres using the $k$-th eigenfunctions.
\end{proposition}
\iffalse
\begin{rem}
In the proof of the above theorem, the density property implies there are infinite many choice of rational points $\{R_1,\cdots,R_m\}$, especially, there are infinite many choice of $\epsilon$ in the construction of matrix data. It follows that for every flat rational $n$-torus $T^n$, there exists infinite many positive integers $k$ such that $T^n$ can be minimally immersed into spheres using the $k$-th eigenfunctions. This generalizes another result of Bryant obtained for $2$-torus in \cite{Bryant}. %admits infinite many minimal isometric immersions in spheres, whose volumes form a unbounded sequence.    
\end{rem}
\fi
\begin{rem}
 Given a minimal rational flat $n$-torus $x: T_1^n=\mathbb{R}^n/\Lambda_n\rightarrow \mathbb{S}^m$ with $\{Q, Y\}$ as its matrix data.  %if $\{e_i\}$ is the generator of dual lattice $\Lambda^*$, then $Q:=(e_i^te_j)\in GL(n,\mathbb{Q})$.
    The standard congruence transformation shows the existence of a rational matrix $B$, such that $\det B=\pm 1$ and $D\triangleq B^tQB$ is diagonal. Decompose $B^{-1}$ as $B^{-1}=F^{-1}\widetilde{B}$, where $F$ is diagonal and $\widetilde{B}, F\in M(n,\mathbb{Z})$. 
    Consider the flat $n$-torus $T_2^n\triangleq \mathbb{R}^n/\tilde{\Lambda}_n$, where $\tilde{\Lambda}_n$ is an orthotope lattice taking $D^{-1}F^{2}$ as its Gram matrix. Then the matrix data $\{D(F^{-1})^2, Y\}$ gives a minimal immersion of $\tilde{x}: T_2^n \rightarrow \mathbb{S}^m$. It is easy to see that $\tilde{\Lambda}_n$ is a sublattice of ${\Lambda}_n$, which implies $\tilde{x}$ is a covering of $x$.  % As a corollary, every minimal flat rational $n$-torus can be seen as a minimal immersion of flat orthotope $n$-torus. 
    Therefore, for every minimal rational flat $n$-torus in spheres, its image can be viewed as the result of a minimal immersion of an orthotope flat $n$-torus. 

\end{rem}

\section{Minimal isometric immersions of %irrational 
flat $3$-tori}\label{sec-3tori}
%In this section, we investigate the minimal immersions of flat $3$-tori, focusing on seeking the irrational ones and non-homogeneous ones. 
In this section, we {focus on} the minimal isometric immersions (embeddings) of flat $3$-tori into spheres. {For irrational cases, we establish an upper bound on the algebraic irrationality degree and demonstrate its optimality through explicit examples. Furthermore, we show all irrational minimal flat 3-tori must be homogeneous, while an  explicit non-homogeneous example in the rational case is provided. %ho   and prove the homogeneity of the minimal immersions. We also construct various examples to show the existence.} {%We first focus on the irrational case to show the existence and rigidity of minimal flat irrational $3$-tori.  
%For rational one, we construct examples of non-homogeneous minimal flat rational $3$-tori.
}
 
\subsection{Irrational minimal flat $3$-tori}\label{subsec-3tori}
 Note that no irrational flat $2$-torus admits minimal isometric immersions into spheres. Our work in \cite{LWX} showed that this remains true for $\lambda_1$-minimal isometric immersions of irrational $3$-tori. 
 %$T^4$. % for $\lambda_1$-minimal isometric immersions into spheres. 
In this subsection, we consider the general minimal isometric immersions of irrational flat $3$-tori. %Several examples will be given. 

For an irrational flat $n$-torus $\mathbb{R}^n/\Lambda_n$, we call it {\em quadratic irrational} (resp. {\em cubic irrational}, {\em quartic irrational})  if the field extension degree $[K:\mathbb{Q}]$ (see Definition~\ref{def-ra}) equals $2$ (resp.\ $3$, $4$). 
\iffalse{We define a flat torus to be {\em quadratic irrational} (respectively {\em cubic irrational}, {\em quartic irrational}) if %there exists a quadratic (resp. cubic, quartic) extension $\mathbb{Q}(w)$ of $\mathbb{Q}$ such that its Gram matrix lies in $GL(n, \mathbb{Q}(w))$ %\setminus 
%but not in $GL(n, \mathbb{Q})$.  
its Gram matrix lies in $GL(n, \mathbb{Q}(w))$ %\setminus 
but not in $GL(n, \mathbb{Q})$, where $\mathbb{Q}(w)$ is a  quadratic (resp. cubic, quartic) extension of  $\mathbb{Q}$. }\fi

We first introduce an approach to construct minimal isometric immersions of quadratic 
irrational flat $3$-tori via Theorem~\ref{thm-vari}. %({in fact the dual version}).%, and an example. 

We begin by selecting a set $Y=\{Y_1, Y_2, \cdots, Y_N\}\subset \mathbb{Z}^3$ such that
$$\mathrm{rank}\,{Y}=3,~~~~~\mathrm{rank}\{Y_1Y_1^t, Y_2Y_2^t, \cdots, Y_NY_N^t\}=5,$$%~~~~~W_Y\neq \emptyset,$$
%and 
%$$\Sp\{Y_1Y_1^t, Y_2Y_2^t, \cdots, Y_NY_N^t\}\cap \Sigma_+\subset \{\sum_{j=1}^N r_j Y_jY_j^t\;|\: r_j\geq 0,~1\leq j\leq N\}.$$
Note that in $\mathrm{Sym}_3$, the line defined by the constraints 
$$\langle Q, Y_jY_j^t\rangle=1, ~~~1\leq j\leq N$$
intersects the hyperplane $\mathrm{Span}\{Y_1Y_1^t, Y_2Y_2^t, \cdots, Y_NY_N^t\}$ at a unique point, which is denoted by $Q_0$. Since such $Q_0$ is completely determined by %can be solved from 
the linear equations
$$\langle \sum_{k=1}^N c_k Y_k Y_k^t, Y_jY_j^t\rangle=1, ~~~1\leq j\leq N,$$
%we see that $Q_0$ 
it is obviously rational. %{In general, the matrix $Q_0$ is non-degenerated.}   

Next,  we choose a rational matrix on the line 
$$\langle Q, Y_jY_j^t\rangle=0, ~~~1\leq j\leq N,$$
and denote it by $Q_1$. Then the hyper-ellipsoids passing through all points in $Y$ constitute a line segment, which can be parameterized by $W_Y=\{Q_0+t\,Q_1\;|\; t\in (a,b)\}$. %$LS\triangleq\{Q_0+t\,Q_1\;|\; t\in (a,b)\}$.  
%({might be empty}).%$a<b\in \mathbb{R}$}. 
If necessary, we rechoose $Q_0$ as $Q_0+t_0 Q_1$ for some rational $t_0\in (a,b)$ so that $Q_0$ is non-degenerate. 

Then, we consider the maximal point of $\det$ restricted on $W_Y$. Observe that $\det(Q_0+t\,Q_1)$,  as a polynomial in $t$, has degree {at most} $3$. Consequently, the point $t_0$ maximizing the determinant function $\det$ on $W_Y$ belongs to $GL(n, \mathbb{Q}(w))$, where $\mathbb{Q}(w)$ is a quadratic extension of $\mathbb{Q}$. %Moreover, t
This maximal point is not rational if and only if the derivative of 
$$\det(Q_0+t\,Q_1)=\det Q_0 \left(1 + \operatorname{tr}(Q_0^{-1}Q_1)\,t + \frac{\operatorname{tr}(Q_0^{-1}Q_1)^2 - \operatorname{tr}(Q_0^{-1}Q_1Q_0^{-1}Q_1)}{2}\,t^2 + \det(Q_0^{-1}Q_1)\,t^3\right)$$ with respect to $t$ has no rational zeros, which is equivalent to say that 
$$\left(\operatorname{tr}(Q_0^{-1}Q_1)^2 - \operatorname{tr}(Q_0^{-1}Q_1Q_0^{-1}Q_1)\right)^2 - 12 \det(Q_0^{-1}Q_1) \operatorname{tr}(Q_0^{-1}Q_1)
$$
is not the square of some rational number. 

Finally, to obtain the desired minimal isometric immersions into spheres we need to verify whether the inverse of $3(Q_0+t_0\,Q_1)$ lies in the convex hull $C_Y$; that is, to studying the existence of non-negative coefficients $\{c_1^2,\cdots,c_N^2\}$ such that 
\begin{equation}\label{eq:flat-2}
\frac{(Q_0+t_0 Q_1)^{-1}}{3}=c_1^2 Y_1Y_1^t+c_2^2Y_2Y_2^t+\cdots\cdots+c_N^2Y_NY_N^t,
\end{equation}
and 
$$c_1^2+c_2^2+\cdots+c_N^2=1.$$
\begin{example}\label{ex-rank5} {A quadratic irrational minimal flat $3$-torus in $\mathbb{S}^9$.} 

Consider the integer set 
$$Y=\left(
\begin{array}{ccccc}
 1 & 0 & 0 & 6 & 6 \\
 0 & 1 & 0 & 12 & 9 \\
 0 & 0 & 1 & -15 & -12 \\
\end{array}
\right).$$
By the above approach, it is straightforward to calculate that 
$$Q_0=\left(
\begin{array}{ccc}
 1 & -\frac{343}{1233} & \frac{397}{1233} \\
 -\frac{343}{1233} & 1 & \frac{1048}{1233} \\
 \frac{397}{1233} & \frac{1048}{1233} & 1 \\
\end{array}
\right),~~~Q_1=\left(
\begin{array}{ccc}
 0 & 10 & 6 \\
 10 & 0 & 1 \\
 6 & 1 & 0 \\
\end{array}
\right),$$
and the maximal point of $\det$ on $%LS\triangleq
W_Y=\{Q_0+t\,Q_1\;|\; t\in \mathbb{R}\}\cap \Sigma_+$ is given by  
$$t_0=\frac{39337-137 \sqrt{10801}}{443880}.$$ 
Moreover, the equation \eqref{eq:flat-2} can be solved by choosing  
$(c_1^2, c_2^2, c_3^2, c_4^2, c_5^2)$ to be 
{\tiny$$\left(\frac{3 \left(12773-107 \sqrt{10801}\right)}{48040},\frac{27 \left(1105-7 \sqrt{10801}\right)}{38432},\frac{3 \left(541 \sqrt{10801}-52459\right)}{192160},\frac{121721+481 \sqrt{10801}}{576480},\frac{7 \left(191 \sqrt{10801}+2791\right)}{576480}\right).$$}
\!\!Then, we obtain a minimal embedding of a flat irrational $3$-torus in $\mathbb{S}^9$, whose matrix data is given by 
    $$Q=\left(
\begin{array}{ccc}
 1 & \frac{39337-137 \sqrt{10801}}{44388}-\frac{343}{1233} & \frac{39337-137 \sqrt{10801}}{73980}+\frac{397}{1233} \\
 \frac{39337-137 \sqrt{10801}}{44388}-\frac{343}{1233} & 1 & \frac{39337-137 \sqrt{10801}}{443880}+\frac{1048}{1233} \\
 \frac{39337-137 \sqrt{10801}}{73980}+\frac{397}{1233} & \frac{39337-137 \sqrt{10801}}{443880}+\frac{1048}{1233} & 1 \\
\end{array}
\right),$$
$$Y=\left(
\begin{array}{ccccc}
 1 & 0 & 0 & 6 & 6 \\
 0 & 1 & 0 & 12 & 9 \\
 0 & 0 & 1 & -15 & -12 \\
\end{array}
\right).$$
The embeddedness property follows directly from Remark~\ref{rk-em}.
\end{example}

Next, we establish an upper bound for the algebraic irrationality degree of minimal flat $3$-tori.   
\begin{theorem}\label{thm-3irra}
   Let $T^3=\mathbb{R}^3/\Lambda_3$ be a flat $3$-torus. If it admits a minimal isometric immersion into some sphere $\mathbb{S}^m$, then the  Gram matrix of $\Lambda_3$ lies in $GL({3}, \mathbb{Q}(w))$%\setminus GL(n, \mathbb{Q})$
   , and the extension degree of $\mathbb{Q}(w)/\mathbb{Q}$ is at most $4$. Moreover, if the minimal isometric immersion of $T^3$ in $\mathbb{S}^m$ is full and $[\mathbb{Q}(w):\mathbb{Q}]=3$ or $4$, then $m=7$ { and the immersion is homogeneous}. 
\end{theorem}
\begin{proof}
Suppose the matrix data associated to the minimal immersion of $T^3$ is $\{Q, Y\}$. Then we have  
$$\rk\,Y=3,~~~3\leq \rk\{Y_jY_j^t\,|\, Y_j\in Y\}\leq6.$$

If $\rk\{Y_jY_j^t\,|\, Y_j\in Y\}=6$, then $Q$ can be determined uniquely from \eqref{eq-AQA}, and thus it must be rational. 

If $\rk\{Y_jY_j^t\,|\, Y_j\in Y\}=5$, then it follows from the approach stated at the beginning of this section that $Q$ is at most quadratic irrational. %{(%note that $LS\neq \emptyset$, 
%if $Q_0$ is degenerated, we replace it by a rational matrix in $W_Y$)}. 

If  $\rk\{Y_jY_j^t\,|\, Y_j\in Y\}=4$, without loss of generality, we assume that 
\begin{equation}\label{eq-assum}
\rk\{Y_1, Y_2, Y_3\}=3,~~~\rk\{Y_1Y_1^t, Y_2Y_2^t, Y_3Y_3^t, Y_4Y_4^t\}=4.
\end{equation} 
Let $P$ 
be the matrix formed by the column vectors $\{Y_1, Y_2, Y_3\}$. Set $\widetilde{Y}\triangleq P^{-1} Y$ and $\widetilde{Q}\triangleq P^tQP$. Note that $\widetilde{Q}$ achieves the maximum of the determinant function on the set $\Sigma_{\widetilde{Y}}$ of all positive definite matrices such that 
\begin{equation}\label{eq-YQ}
\widetilde{Y}_j^t \widetilde{Q} \widetilde{Y}_j=1, ~~~1\leq j\leq \sharp(Y).
\end{equation}
%We first claim that $\sharp(Y)=4$. In fact if 

We assume that $\widetilde{Y}_4=(r_1, r_2, r_3)^t$ and 
$\widetilde{Y}_j=(s_{j_1}, s_{j_2}, s_{j_3})^t$ for $5\leq j\leq \sharp(Y)$. Then it follows from 
$$\widetilde{Y}_j\widetilde{Y}_j^t=\lambda_{j_1} \widetilde{Y}_1\widetilde{Y}_1^t+\lambda_{j_2} \widetilde{Y}_2\widetilde{Y}_2^t+\lambda_{j_3} \widetilde{Y}_3\widetilde{Y}_3^t+\lambda_{j_4} \widetilde{Y}_4\widetilde{Y}_4^t$$%~~~5\leq j\leq \sharp(Y)$$
that 
%$$\lambda_{j_1}=s_{j_1}^2, ~~~\lambda_{j_2}=s_{j_2}^2,~~~\lambda_{j_3}=s_{j_3}^2,~~~$$
\beq \label{eq-rs}
s_{j_1}s_{j_2}=\lambda_{j_4} r_1r_2,~~~s_{j_1}s_{j_3}=\lambda_{j_4} r_1 r_3,~~s_{j_2}s_{j_3}=\lambda_{j_4} r_2 r_3.
\eeq
%If one of $\{r_1, r_2, r_3\}$ 
Due to \eqref{eq-assum}, we observe that $\{r_1, r_2, r_3\}$ contains at least two nonzero elements. Without loss of generality, we assume that $r_1\neq 0$ and $r_3\neq 0$. If $r_2= 0$, then %by \eqref{eq:flat} 
it is straightforward to verify that all matrices { in $C_{\widetilde{Y}}$} must necessarily exhibit a block-diagonal  structure, which implies the maximum point $\widetilde{Q}$ must be rational. 
%verify that $\widetilde{Q}^{-1}$, and hence $\widetilde{Q}$, must necessarily exhibit a block-diagonal structure. Consequently, as the maximum point, $\widetilde{Q}$} must be rational.
%then it is straightforward to verify that {\color{red} all matrices satisfying \eqref{eq-YQ} must necessarily exhibit a block-diagonal structure, which implies} the maximum point $\widetilde{Q}$ must be rational. 

Next, we consider the case $r_2\neq 0$.  
If $\lambda_{j_4}=0$, then $\widetilde{Y}_j$ is parallel to one of $\{\widetilde Y_1, \widetilde Y_2, \widetilde Y_3\}$; hence it must equal one of $\{\pm\widetilde Y_1, \pm\widetilde Y_2, \pm\widetilde Y_3\}$ by \eqref{eq-YQ}. %For the case of 
If $\lambda_{j_4}\neq 0$, %from \eqref{eq-assum}, we observe that $\{r_1, r_2, r_3\}$ contains at least two non-zero elements. Without loss of generality, we assume that $r_1\neq 0$ and $r_3\neq 0$. 
then using \eqref{eq-rs} we obtain that  $s_{j_2}r_3=r_2s_{j_3}, \,s_{j_1}r_2=r_1s_{j_2},$ and $s_{j_1}r_3=r_1s_{j_3}$, %. If $r_2\neq 0$, we also have  that $s_{j_1}r_3=r_1s_{j_3}$, 
which implies $\widetilde{Y}_j// \widetilde{Y}_4$; hence $\widetilde{Y}_j=\pm\widetilde{Y}_4$ by \eqref{eq-YQ}. Therefore, we derive that $\sharp(Y)=4$ in this case. %of $r_2\neq0$. For the case of $r_2= 0$, then $s_{j_2}=0$.
%Moreover, we can 
We parameterize the matrices in $W_{\widetilde{Y}}$ as 
\begin{equation}\label{eq:Q}
\left(
\begin{array}{ccc}
 1 & a & b \\
 a & 1 & c \\
 b & c & 1 \\
\end{array}
\right),
\end{equation}
with $a, b, c$ satisfying  
$$2 a\, r_1 r_2 +2 b\,  r_1 r_3 +2 c\, r_2  r_3 +r_1 ^2+r_2 ^2+r_3 ^2=1,~~|a|<1,~~~|b|<1,~~~|c|<1.$$
The determinant of such matrices can be expressed as 
$$2 a b c- a^2 - b^2 - c^2+1.$$ 
Using the method of Lagrange multipliers, the critical point of the determinant function on $W_{\widetilde{Y}}$ satisfies: %can be characterized as follows, %$\widetilde{Q}$ 
\begin{equation}\label{eq-bc}
b=-\frac{(a r_2+r_1) \left(2 a r_1 r_2+r_1^2+r_2^2+r_3^2-1\right)}{2 r_3 \left(2 a r_1 r_2+r_1^2+r_2^2\right)},\quad c=-\frac{(a r_1+r_2) \left(2 a r_1 r_2+r_1^2+r_2^2+r_3^2-1\right)}{2 r_3 \left(2 a r_1 r_2+r_1^2+r_2^2\right)},
\end{equation}
%with $a$ satisfying 
\begin{equation}\label{eq-critic}\begin{aligned}
\!\!\!\!\!r_1 r_2 \left((r_1^2+r_2^2)^2 -\left(r_3^2-1\right)^2\right)-\left(r_1^2+r_2^2\right) \left(r_1^4-2 r_1^2 \left(r_2^2+r_3^2+1\right)+r_2^4-2 r_2^2 \left(r_3^2+1\right)+\left(r_3^2-1\right)^2\right)a\\
- r_1 r_2 \left(7 r_1^4+2 r_1^2 \left(5 r_2^2-4 \left(r_3^2+1\right)\right)+7 r_2^4-8 r_2^2 \left(r_3^2+1\right)+\left(r_3^2-1\right)^2\right)a^2\\
-8 r_1^2 r_2^2 \left(2 r_1^2+2 r_2^2-r_3^2-1\right)\,a^3-12 r_1^3 r_2^3\,a^4=0
\end{aligned}
\end{equation}
Since $r_1, r_2, r_3\in \mathbb{Q}$, it follows that there exists a field extension $\mathbb{Q}(w)$ of $\mathbb{Q}$ such that the extension degree $[\mathbb{Q}(w):\mathbb{Q}]$ is at most $4$ and $a,b,c \in \mathbb{Q}(w)$. %We will call the polynomial appearing in \eqref{eq-critic} a critical polynomial.  
%$\det$ satisfies $a,b,c \in \mathbb{Q}(w)$, and the extension degree $[\mathbb{Q}(w):\mathbb{Q}]$ is at most $4$. 

If  $\rk\{Y_jY_j^t\,|\, Y_j\in Y\}=3$, then following the same argument as above, up to multiplying a rational matrix on the left, we may assume that 
\begin{equation*}%\label{eq-assum}
%\rk\{Y_1, Y_2, Y_3\}=3,~~~\rk\{Y_1Y_1^t, Y_2Y_2^t, Y_3Y_3^t, Y_4Y_4^t\}=4. 
Y_1=(1,0,0),~~~Y_2=(0,1,0),~~~Y_3=(0,0,1). 
\end{equation*} 
Therefore, in this case the maximizer of $\det$ on $W_Y$ is $GL(n,\mathbb{Q})$-congruent to the identity matrix.  

For the second part of the theorem, it follows from the above argument that %in the case 
$\rk\{Y_jY_j^t\,|\, Y_j\in Y\}={4}$. {It is easy to verify that any $\eta$-set in $Y$ satisfies Lemma~2.3 in \cite{LWX}, which implies the minimal isometric immersion must be homogeneous.}  
\iffalse{$\widetilde{Y}_5=(s_1, s_2, s_3)^t$, then it follows from 
$$\widetilde{Y}_5\widetilde{Y}_5^t=\lambda_1 \widetilde{Y}_1^t\widetilde{Y}_1+\lambda_2 \widetilde{Y}_2^t\widetilde{Y}_2+\lambda_3 \widetilde{Y}_3^t\widetilde{Y}_3+\lambda_4 \widetilde{Y}_4^t\widetilde{Y}_4$$
that 
%$$\lambda_1=s_1^2, ~~~\lambda_2=s_2^2,~~~\lambda_3=s_3^2,~~~$$
\beq \label{eq-rs}
s_1s_2=\lambda_4 r_1r_2,~~~s_1s_3=\lambda_4 r_1 r_3,~~s_2s_3=\lambda_4 r_2 r_3.
\eeq
If $\lambda_4=0$, then $\widetilde{Y}_5$ is parallel to one of $\{\widetilde Y_1, \widetilde Y_2, \widetilde Y_3\}$; hence must be equal to one of $\{\pm\widetilde Y_1, \pm\widetilde Y_2, \pm\widetilde Y_3\}$ by \eqref{eq-YQ}. For the case of $\lambda_4\neq 0$, from \eqref{eq-assum}, we observe that $\{r_1, r_2, r_3\}$ contains at least two non-zero elements. Without loss of generality, we assume that $r_1\neq 0$ and $r_3\neq 0$. Then using \eqref{eq-rs} we obtain that  $s_2r_3=r_2s_3, s_1r_2=r_1s_2$. If $r_2\neq 0$, we also have  that $s_1r_3=r_1s_3$, which implies $\widetilde{Y}_5// \widetilde{Y}_4$; hence $\widetilde{Y}_5=\pm\widetilde{Y}_4$ by \eqref{eq-YQ}. If $r_2= 0$, then $s_2=0$.} 
\fi
\end{proof}
\begin{rem}\label{rem2}
    The proof of Theorem \ref{thm-3irra} provides a method to construct examples in $\mathbb{S}^7$.  
First, choose a rational set $Y$ of the following form 
\[
Y_{1}=(1,0,0),\qquad 
Y_{2}=(0,1,0),\qquad 
Y_{3}=(0,0,1),\qquad 
Y_{4}=(r_{1},r_{2},r_{3}).
\]
Next, compute the roots of equation \eqref{eq-critic} and construct matrices $Q$ of the form \eqref{eq:Q}, where the parameters $b$ and $c$ are given in \eqref{eq-bc}.  
For such a $Q$, if $Q$ is positive definite and there exist positive coefficients $c_{1}^{2},c_{2}^{2},c_{3}^{2},c_{4}^{2}$ satisfying
\begin{equation}\label{eq-cijY}
c_{1}^{2}+c_{2}^{2}+c_{3}^{2}+c_{4}^{2}=1,~~~c_{1}^{2}Y_{1}Y_{1}^{t}+c_{2}^{2}Y_{2}Y_{2}^{t}+c_{3}^{2}Y_{3}Y_{3}^{t}+c_{4}^{2}Y_{4}Y_{4}^{t}
\;=\;\frac{1}{3}\,Q^{-1},
\end{equation}
then, after a dilation, the pair $\{Q, Y\}$ yields the matrix data of a minimal flat $3$-torus in $\mathbb{S}^{7}$. 
%Next, we show that for $\rk\{Y_jY_j^t\,|\, Y_j\in Y\}=4$, there do exist 
\end{rem}
%Next, we show that for $\rk\{Y_jY_j^t\,|\, Y_j\in Y\}=4$, there do exist examples with all  possible algebraic irrationality degree.
By this method, we construct examples of irrational minimal flat $3$-tori in $\mathbb{S}^{7}$ with all possible algebraic irrationality degree. 

\begin{example}\label{ex-4.4}{A quadratic irrational flat $3$-torus into $\mathbb{S}^7$.}

Consider the integer set 
$$
Y=\left(
\begin{array}{cccc}
 1 & 0 & 0 & -3 \\
 0 & 1 & 0 & 4 \\
 0 & 0 & 1 & -3 \\
\end{array}
\right).%~~~
$$
Now the equation \eqref{eq-critic} is equivalent to  
$$f(x)\triangleq -9 \left(32 x^2-20 x-17\right) \left(72 x^2-115 x+44\right)=0.$$
One of its roots is given by $a = \frac{1}{144} \bigl(115 - \sqrt{553}\bigr)$, which determines the following positive definite matrix  
by \eqref{eq-bc}, 
$${ Q=\left(
\begin{array}{ccc}
 1 & \left(115-\sqrt{553}\right)/144 &  \left(16-\sqrt{553}\right)/54 \\
  \left(115-\sqrt{553}\right)/144 & 1 &  \left(115-\sqrt{553}\right)/144 \\
  \left(16-\sqrt{553}\right)/{54} &  \left(115-\sqrt{553}\right)/144 & 1 \\
\end{array}
\right)}. 
$$
It is straightforward to verify that 
$$ (c_1^2,\, c_2^2,\, c_3^2,\, c_4^2)=(38-\sqrt{553})\left(\frac{2}{99},\frac{\sqrt{553}-13}{1782},\frac{2}{99},\frac{\sqrt{553}+17}{1782}\right)$$
solve the equations in \eqref{eq-cijY}. 
Due to  Remark~\ref{rem2}, %obtain 
the matrix data $\{Q, Y\}$ %of minimal flat $3$-torus in $\mathbb{S}^{7}$. 
gives a minimal  embedding of a quadratic irrational flat $3$-torus into $\mathbb{S}^7$, where the embeddedness property follows directly from Remark~\ref{rk-em}.  

This example shows that $\rk\{Y_jY_j^t\,|\, Y_j\in Y\}=5$ is not a necessary condition to produce a quadratic irrational flat $3$-torus allowing minimal immersion into $\mathbb{S}^m$. 
\end{example}

\begin{example} {A cubic irrational minimal flat $3$-torus in $\mathbb{S}^7$.}

Consider the integer set  
$$
Y=\left(
\begin{array}{cccc}
 4 & 0 & 0 & -5 \\
 0 & 4 & 0 & 2 \\
 0 & 0 & 4 & -3 \\
\end{array}
\right).%~~~
$$
Now the equation \eqref{eq-critic} is equivalent to  
$$f(x)\triangleq(50 x^3-160 x^2+149 x-33)(x+1)=0.$$
Note that $f(0)f(\frac{1}{2})<0$, so there exists $a\in [0,\frac{1}{2}]$ as a real root of $f(x)$.
By \eqref{eq-bc}, set 
%we find a positive definite matrix % It follows from the proof of Theorem~\ref{thm-3irra} that, the maximal point of the determinant function restricted on $W_Y$ is 
%\vskip 0.2cm
$${ Q\triangleq\frac{1}{16}\left(
\begin{array}{ccc}
 1 & a & \frac{(2 a-5) (10 a-11)}{3 (20 a-29)} \\
 a & 1 & -\frac{(5 a-2) (10 a-11)}{3 (20 a-29)} \\
 \frac{(2 a-5) (10 a-11)}{3 (20 a-29)} & -\frac{(5 a-2) (10 a-11)}{3 (20 a-29)} & 1 \\
\end{array}
\right)}.
$$
%where %$a\approx0.321061$ 
%$a$ 
%is a root of the %following 
%polynomial \hlt{(arising from \eqref{eq-critic})}%(irreducible over $\mathbb{Q}$, can be checked directly on  www.wolframalpha.com)
%$$f(x)=(50 x^3-160 x^2+149 x-33)(x+1).$$
%so that $Q$ is positive definite. 
%Note that $a\neq -1$. In fact, such root exists in $[0,\frac{1}{2}]$, since $50 x^3-160 x^2+149 x-33$ takes different sign at $0$ and $\frac{1}{2}$. It is easy to verify that 
%Note that $f(0)f(\frac{1}{2})<0$, so there exists at least one real root $a_0\in (0,\frac{1}{2})$ of $f(x)$. Meanwhile, 
It is straightforward to check that {at} $a\in [0,\frac{1}{2}]$ 
$$\det(Q)=\frac{20 (a+1)^2 (5a-7)(a-1) }{9 (29-20 a)}>0,$$
hence $Q$ is positive definite. Furthermore, one can also verify that the following coefficients 
\begin{align*}
    &c_1^2=%\frac{-(25 a^2-95 a+76)}{6 (a^2-1)  (5 a-7) (20 a-29)}=
    \frac{57-(10 a - 19)^2 }{24 (a^2-1)  (5 a-7) (20 a-29)},&
    &c_2^2=\frac{-(430 a^2-1235 a+883)}{15 (a^2-1)  (5 a-7) (20 a-29)},\\
    &c_3^2=\frac{3 (5 a-9)}{10 (a+1) (5 a-7)},&
    &c_4^2=\frac{4 (10 a-11)}{15 (a+1) (5 a-7)}
\end{align*}
are positive and the equations in \eqref{eq-cijY} hold true. 

Due to  Remark~\ref{rem2}, we obtain a matrix data $\{Q, Y\}$ of minimal flat $3$-torus in $\mathbb{S}^{7}$.
%positive in $[0, \frac{1}{2}]$ and solve \eqref{eq-cijY}. 
%The uniqueness of such root follows from the convexity of $\det$ on the cone of positive definite matrices. %the determinant of $Q$ at the endpoints $a=0$ and $a=\frac{1}{2}$ is positive, which implies there must exist a root of   the polynomial  characterizes the critical 
%$$f(0.3)=$$
%Numerically, we find %One can verify directly that 
%$a\approx0.321061$.} 
%over $\mathbb{Q}$ on a minimal polynomial 
%It is straightforward to verify that 
%\hlt{This polynomial is irreducible over $\mathbb{Q}$, which can be checked directly on WolframAlpha by "IrreduciblePolynomialQ[f(x)]"}.
%with numerical approximation $a\approx0.321061$. 
% Then, by taking 
% \begin{align*}
%     &c_1^2=\frac{-(25 a^2-95 a+76)}{6 (a^2-1)  (5 a-7) (20 a-29)}\approx 0.0733429,\\
%     &c_2^2=\frac{-(430 a^2-1235 a+883)}{15 (a^2-1)  (5 a-7) (20 a-29)}\approx0.323914,\\
%     &c_3^2=\frac{3 (5 a-9)}{10 (a+1) (5 a-7)}\approx0.31128,\\
%     &c_4^2=\frac{4 (10 a-11)}{15 (a+1) (5 a-7)}\approx0.291462,
% \end{align*}
% we obtain a solution to equation \eqref{eq:flat}.  
Note that  $a$ is a root of 
$$50 x^3-160 x^2+149 x-33,$$ which is irreducible\footnote{It can be verified by the command ``IrreduciblePolynomialQ[$\cdot$]" on WolframAlpha: www.wolfram.com} in $\mathbb{Q}$. %(can be verified by the command ``IrreduciblePolynomialQ[$\cdot$]" on WolframAlpha).
Therefore, %we obtain 
%\vskip 0.2cm
%\begin{small}$$ (a_1,\, a_2,\, a_3,\, a_4)=(38-\sqrt{553})\left(\frac{2}{99},\frac{\sqrt{553}-13}{1782},\frac{\sqrt{553}+17}{1782},\frac{2}{99}\right),$$
%\end{small}
%gives 
the minimal flat $3$-torus we obtained is cubic irrational.
%a cubic irrational flat $3$-torus into $\mathbb{S}^7$. {%While 

Unfortunately, this immersion fails to be an embedding. In fact, the following points 
$$(0, \frac{1}{4},\frac{1}{2}),~~~(0, \frac{1}{2},0),~~~(\frac{1}{4},0,\frac{1}{4})$$ 
are all mapped to nonzero integer points under \eqref{eq-em}. Regarding them as row vectors, we obtain a matrix 
\begin{equation}\label{eq-newY}
\widetilde{Y}=\left(
\begin{array}{cccc}
 0 & 1 & 2 & -1 \\
 0 & 2 & 0 & 1 \\
 1 & 0 & 1 & -2 \\
\end{array}
\right).
\end{equation}
Set 
$$\widetilde{Q}\triangleq\left(
\begin{array}{ccc}
 -2 & 0 & 2 \\
 1 & 2 & -1 \\
 4 & 0 & 0 \\
\end{array}
\right)Q\left(
\begin{array}{ccc}
 -2 & 0 & 2 \\
 1 & 2 & -1 \\
 4 & 0 & 0 \\
\end{array}
\right)^{t}= {\frac{1}{16}}\left(
\begin{array}{ccc}
 \frac{16 \left(10 a^2-66 a+71\right)}{87-60 a} & -\frac{40 \left(3 a^2-a-4\right)}{20 a-29} & \frac{16 \left(10 a^2-66 a+71\right)}{60 a-87} \\
 -\frac{40 \left(3 a^2-a-4\right)}{20 a-29} & \frac{16 \left(25 a^2-9 a-34\right)}{60 a-87} & \frac{8 \left(50 a^2-21 a-71\right)}{60 a-87} \\
 \frac{16 \left(10 a^2-66 a+71\right)}{60 a-87} & \frac{8 \left(50 a^2-21 a-71\right)}{60 a-87} & 16 \\
\end{array}
\right).$$
Then the matrix data $\{\widetilde{Q},\widetilde{Y}\}$ gives a minimal embedding of a cubic irrational flat 3-torus, which is covered by the aforementioned immersed example. Here the embeddedness property follows directly from an analysis of \eqref{eq-em} and a verification that $\mathcal{Y}(\Omega)$ contains no nonzero integer points.
\end{example}

\begin{example}\label{ex-rank4-2} {A quartic irrational minimal flat $3$-torus in $\mathbb{S}^7$.}

Consider the matrix data 
$$
Y=\left(
\begin{array}{cccc}
 1 & 0 & 0 & 5 \\
 0 & 1 & 0 & 7 \\
 0 & 0 & 1 & 8 \\
\end{array}
\right).%~~~
$$
Now the equation \eqref{eq-critic} is equivalent to  
$$f(x)\triangleq-14700 x^4-23240 x^3+1079 x^2+10730 x+1507=0.$$
Note that $f(-\frac{3}{20})f(-\frac{1}{7})<0$, so $f(x)$ has a real root $a\approx-0.149201\in [-\frac{3}{20},-\frac{1}{7}]$.
By \eqref{eq-bc}, set 
$${ Q=
\begin{pmatrix}
 1 & a & {-\frac{(7 a+5) (70 a+137)}{32 (35 a+37)}} \\
 a & 1 & {-\frac{(5 a+7) (70 a+137)}{32 (35 a+37)}} \\
{ -\frac{(7 a+5) (70 a+137)}{32 (35 a+37)}}& {-\frac{(5 a+7) (70 a+137)}{32 (35 a+37)}} & 1 \\
\end{pmatrix}
}.
$$
It is straightforward to check that at  $a\in [-\frac{3}{20},-\frac{1}{7}]$, 
$$\det(Q)=\frac{35 (a^2-1) (10a-1) (14 a+5)}{512 (37 + 35 a)}>0,$$
hence $Q$ is positive definite. %Again, 
Furthermore, one can also verify that the following coefficients 
%with numerical approximation $a\approx-0.149201$. 
%Then, by taking 
\begin{align*}
 &c_1^2=   {\frac{24500 a^4-16800 a^3-107125 a^2-67998 a-4945}{42 (a-1) (a+1) (10 a-1) (14 a+5) (35 a+37)}},%\approx0.30459,
    \\
    &c_2^2={\frac{34300 a^4+23520 a^3-88151 a^2-84090 a-8699}{30 (a-1) (a+1) (10 a-1) (14 a+5) (35 a+37)}},\\%\approx0.26971,\\
    &c_3^2=\frac{-128 (70 a+11)}{105 (10 a-1) (14 a+5)},\\%\approx0.0934204,\\
    &c_4^2=\frac{-2 (70 a+137)}{105 (10 a-1) (14 a+5)} %\approx0.332279,
\end{align*}
satisfy the equations in \eqref{eq-cijY}. We claim that they are all positive at $a$. 

First, consider the polynomials, and define
{\begin{align*}
    g_1(x) &\triangleq 24500x^4 - 16800x^3 - 107125x^2 - 67998x - 4945, \\
    g_2(x) &\triangleq 34300x^4 + 23520x^3 - 88151x^2 - 84090x - 8699.
\end{align*}
A straightforward computation yields that, for both  $i = 1, 2$, 
\[
    \lim_{x\to-\infty} g_i(x) = +\infty, \quad g_i(-1) < 0, \quad g_i\left(-\frac{1}{2}\right) > 0, \quad g_i\left(-\frac{1}{7}\right) > 0, \quad g_i(0) < 0, \quad \lim_{x\to+\infty} g_i(x) = +\infty.
\]
By the Intermediate Value Theorem, the alternating signs indicate that the four real roots of each degree-4 polynomial $g_i(x)$ are completely isolated within the intervals $(-\infty, -1)$, $\left(-1, -\frac{1}{2}\right)$, $\left(-\frac{1}{7}, 0\right)$, and $(0, +\infty)$.}
Consequently, no real roots exist within the interval $\left[-\frac{1}{2}, -\frac{1}{7}\right]$. Since the endpoint values are positive, both $g_1(x)$ and $g_2(x)$ are strictly positive in $\left[-\frac{1}{2}, -\frac{1}{7}\right]$.
\iffalse{First, consider the numerators of $c_1^2$ and $c_2^2$, and define 
\[
g_1(x)\triangleq41650 x^3+78995 x^2+37586 x+1825,\quad g_2(x)\triangleq 23030 x^3+64225 x^2+44290 x+3887 .
\]
A straightforward computation yields  
$$\lim_{x\to -\infty}g_i(x)=- \infty,~~~g_i(-1)>0,~~~g_i(-\frac{1}{2})<0,~~~g_i(-\frac{1}{7})<0,~~~\lim_{x\to+ \infty}g_i(x)=+ \infty,~~~i=1,2.$$
So both $g_1(x)$ and $g_2(x)$ are negative in $[-\frac{1}{2},-\frac{1}{7}]$. }\fi%admit only one root in $(-\frac{1}{2},+\infty)$. It is directly to check that $g_i(-\frac{3}{20})<0$, $g_i(-\frac{1}{7})<0$. 
Therefore, %at  so in $(-\frac{3}{20},-\frac{1}{7})\subset (-\frac{1}{2},+\infty)$, 
we have $g_i(a)>0$, which yields that both $c_1^2$ and $c_2^2$ are positive. Finally, one can readily verify that $c_3^2$, $c_4^2$ are also positive. %and \eqref{eq-cijY} hold true. 
By Remark~\ref{rem2}, we obtain a matrix data $\{Q, Y\}$ of minimal flat $3$-torus in $\mathbb{S}^{7}$. 

Note that  $f(x)$ is irreducible\footnote{It can be verified by the command ``IrreduciblePolynomialQ[$\cdot$]" on WolframAlpha: www.wolfram.com} in $\mathbb{Q}$.
Therefore, %we obtain 
%\vskip 0.2cm
%\begin{small}$$ (a_1,\, a_2,\, a_3,\, a_4)=(38-\sqrt{553})\left(\frac{2}{99},\frac{\sqrt{553}-13}{1782},\frac{\sqrt{553}+17}{1782},\frac{2}{99}\right),$$
%\end{small}
%gives 
this example is a minimal immersion of a quartic irrational flat $3$-torus. %The polynomial 
%$$675 x^3+765 x^2-291 x-253$$ can be verified irreducible. Thus, the minimal flat $3$-torus we obtained is cubic irrational.
%in \eqref{eq:flat}, we obtain 
%\vskip 0.2cm
%\begin{small}$$ (a_1,\, a_2,\, a_3,\, a_4)=(38-\sqrt{553})\left(\frac{2}{99},\frac{\sqrt{553}-13}{1782},\frac{\sqrt{553}+17}{1782},\frac{2}{99}\right),$$
%\end{small}
%gives 
%a minimal embedding of a cubic irrational flat $3$-torus into $\mathbb{S}^7$. 
Again by Remark~\ref{rk-em},%analyzing  \eqref{eq-em} and demonstrating the absence of non-zero integer points in $\mathcal{Y}(\Omega)$, one can verify that 
this immersion is actually an embedding.%This immersion is not an embedding. 
\iffalse{The polynomial 
$$-14700 x^4-23240 x^3+1079 x^2+10730 x+1507$$ is still irreducible. Thus, the minimal flat $3$-torus we obtained is quartic irrational.

in \eqref{eq:flat} we obtain 
%\vskip 0.2cm
%\begin{small}$$ (a_1,\, a_2,\, a_3,\, a_4)=(38-\sqrt{553})\left(\frac{2}{99},\frac{\sqrt{553}-13}{1782},\frac{\sqrt{553}+17}{1782},\frac{2}{99}\right),$$
%\end{small}
%gives 
a minimal embedding of a quartic irrational flat $3$-torus into $\mathbb{S}^7$, where the embeddedness property follows directly from Remark~\ref{rk-em}. 
}\fi
\end{example}
\subsection{The uniqueness of minimal immersion for irrational flat $3$-torus}

 As shown in Proposition \ref{eq-infinite}, for any given rational flat $n$-torus $T^n$, there exist infinitely many $k\in\mathbb{Z}^+$ such that $T^n$ can be homothetically and minimally immersed into spheres by the $k$-th eigenfunctions. It is a natural question to ask, how many such  immersions an irrational torus can admit into spheres? In the $3$-dimensional case, we show that such an immersion is unique if the  algebraic irrationality degree exceeds $2$.
\begin{theorem}
\label{prop-irrational}
    For every cubic and quartic irrational flat $3$-torus, if a minimal homothetic immersion into spheres exists, then it is unique up to congruence.  %If an irrational flat $3$-torus has extension degree $[\mathbb{Q}(w):\mathbb{Q}]>2$, and it can be isometrically, minimally immersed into spheres, then the immersion is full and unique into $\mathbb{S}^7$. 
\end{theorem}
\begin{proof}
    Suppose $\{Q,Y\}$ is the matrix data for an immersion of such torus. Then it follows from Theorem~\ref{thm-3irra} that $Y=\{Y_1,Y_2,Y_3,Y_4\}$.  %We know the rank of $\{Y_iY_i^t|1\leq i\leq N\}$ is $4$. Let $Y_1,Y_2,Y_3,Y_4$ be integer vectors in $Y$ such that $\{Y_iY_i^t|1\leq i\leq4\}$ is linearly independent. Since the rank is $4$, 
     Furthermore, %So each $Y_i$ must lie on one of these four lines which implies a full immersion into $\mathbb{S}^7$ and 
    there exist positive real numbers $c_1, c_2, c_3, c_4$ such that 
    \begin{equation}\label{eq-uniq}
    c_1Y_1Y_1^t+c_2Y_2Y_2^t+c_3Y_3Y_3^t+c_4Y_4Y_4^t=\frac{Q^{-1}}{3},\quad \sum_{i=1}^4c_i=1.
    \end{equation}

    Suppose a matrix data $\{\frac{Q}{\lambda^2},X\}$ gives another minimal immersion of this $3$-torus, where $X=\{X_i,1\leq i\leq 4\}$. Let $4+d=\mathrm{rank}\{X_jX_j^t,Y_iY_i^t|1\leq i,j\leq 4\}$, then
    \[
    \mathrm{dim}(\mathrm{Span}\{X_jX_j^t\}\cap\mathrm{Span}\{Y_iY_i^t\})=4-d.
    \]
 
    If $\lambda^2$ is irrational, we assume $X_1X_1^t,\cdots,X_dX_d^t,Y_1Y_1^t,\cdots,Y_4Y_4^t$ are $4+d$ linearly independent vectors in $\mathrm{Sym}_3$. For any $d<\alpha\leq 4$, we have
    \[
    X_\alpha X_\alpha^t=\sum_{i=1}^4 k_{i\alpha}Y_iY_i^t+\sum_{j=1}^d h_{j\alpha}X_jX_j^t,
    \]
    where all of $k_{i\alpha}$ and $h_{j\alpha}$ are obviously rational. Then from $\langle Q,Y_iY_i^t\rangle=1$, $\langle Q,X_jX_j^t\rangle=\lambda^2$,
    \begin{equation}\label{eq-zero}
    \sum_{i=1}^4 k_{i\alpha}=1-\sum_{j=1}^d h_{j\alpha}=0.
    \end{equation}
    So
    \[
    X_\alpha X_\alpha^t-\sum_{j=1}^d h_{j\alpha}X_jX_j^t=\sum_{i=1}^4 k_{i\alpha}Y_iY_i^t,\quad d+1\leq \alpha\leq 4
    \]
    are exactly $4-d$ linearly independent vectors in $\mathrm{Span}\{X_jX_j^t\}\cap\mathrm{Span}\{Y_iY_i^t\}$. Due to \eqref{eq-zero},  $Q$ is orthogonal to $\mathrm{Span}\{X_jX_j^t\}\cap\mathrm{Span}\{Y_iY_i^t\}$. However, \eqref{eq-uniq} shows $Q^{-1}\in \mathrm{Span}\{Y_iY_i^t\}$ and $\langle Q,Q^{-1}\rangle=3\not=0$, which suggests  $Q^{-1}\not\in \mathrm{Span}\{X_jX_j^t\}$. %Thus, the pair $\left{ \frac{Q}{\lambda^2}, X \right}$ cannot induce a minimal immersion into spheres, yielding a contradiction. 
    So $\{ \frac{Q}{\lambda^2},X\}$ cannot give a minimal immersion into spheres, yielding a contradiction. 
    
    Hence, $\lambda^2$ is rational. In this case, applying the similar argument as the beginning of subsection 4.1, we conclude that the maximum of the determinant function on
    \[
    \{\widehat{Q}\in\Sigma_+\subset \mathrm{Sym}_3 \mid \langle\widehat{Q},Y_iY_i^t\rangle=\langle\widehat{Q},\frac{1}{\lambda^2}
    X_jX_j^t\rangle=1, 1\leq i,j\leq 4\}
    \]
    has extension degree at most $2$ provided $d>0$. Since this maximum cannot exceed the maximal value of the determinant on $W_Y$ (which is attained by $Q$), we conclude that $Q$ itself must be the maximizer. By our initial assumption, $Q$ is either a cubic or quartic form.  %Since such the maximum is not greater than the maximum of the determinant on $W_Y$, which is realized by $Q$. Hence, we conclude that such the maximum point must be $Q$, which, by assumption, is either cubic or quartic.   %$Q$ achieves the maximum of determinant on $\langle\widehat{Q},Y_iY_i^t\rangle=1$ and it is cubic or quartic, 
    Hence, we derive that $d=0$ and then %which implies 
    $\mathrm{Span}\{X_jX_j^t\}=\mathrm{Span}\{Y_iY_i^t\}$. %By Theorem \ref{thm-3irra}, 
    We choose %there are 
    $Q_1,Q_2\in \mathrm{Sym}_3$ such that $\mathrm{Span}\{Q_1,Q_2\}\perp\mathrm{Span}\{Y_iY_i^t,1\leq i\leq 4\}$. Note that $x^tQ_1x=0$ and $x^tQ_2x=0$ define two quadratic cones whose intersection is composed by $4$ non-coplanar lines (spanned by $Y_1, Y_2, Y_3, Y_4$, respectively).
    Therefore, $\{\pm Y_1, \pm Y_2, \pm Y_3, \pm Y_4\}$ constitute the whole intersection of the hyper-ellipsoid defined by $Q$ and the aforementioned two quadratic cones. Consequently, we deduce that $X\subset\{\pm\lambda Y_i\}$, and thus $\{ \frac{Q}{\lambda^2},X\}$ is conformally equivalent to $\{Q,Y\}$. {Then the uniqueness follows from the homogeneity established in Theorem~\ref{thm-3irra} and } 
    $${\mathrm{rank}\{Y_1Y_1^t, Y_2Y_2^t, Y_3Y_3^t, Y_4Y_4^t\}=4.}$$
    \vskip -0.5cm
\end{proof}
\begin{rem}
As for the quadratic case, it seems to be very subtle for the uniqueness and the rigidity of minimal immersions. %is very subtle. %In comparison with Theorem~\ref{prop-irrational}, it is natural to ask whether the uniqueness result extends to the quadratic case. 
Let $\{Q, Y\}$ be the matrix data for a minimal immersion of a quadratic irrational flat 3-torus with $\rk\{Y_jY_j^t \mid Y_j\in Y\}=5$. %{\color{red} When $\rk\{Y_jY_j^t,|, Y_j\in Y\}\leq 4$, the proof of Theorem~\ref{prop-irrational} shows that the minimal isometric immersion of this torus is unique.} 
%We only consider the generic case, i.e., $\rk\{Y_jY_j^t,|, Y_j\in Y\}=5$. %However, the situation becomes more subtle when $\rk\{Y_jY_j^t,|, Y_j\in Y\}=5$. 
It turns out that there may exist many rational points on the hyper-ellipsoid $\mathcal{Q}$ determined by $u^t Q u=1$ so that new minimal immersions can be constructed. In fact, given a new rational point $Z$, let $l$ be a common multiple of the denominators of its coordinates. Note that $ZZ^t\in {\mathrm{Span}}\{Y_jY_j^t \mid Y_j\in Y\}$, since $Q$ is irrational. %If $Z$ lies outside the convex hull of $\{Y_j Y_j^t \mid Y_j\in Y\}$, then $\{\frac{Q}{l^2}, lY\cup lZ\}$  provides new minimal immersions. 
For each face $F$ of the convex hull of $\{Y_j Y_j^t \mid Y_j\in Y\}$, consider the inverse cone of $F$ with respect to  $\frac{Q^{-1}}{3}$. Then there exists at least one such inverse cone containing $ZZ^t$, and we denote its corresponding face by $\widetilde{F}$. Define $\widetilde{Y}$ to be the set consisting of the vertices of $\widetilde{F}$ together with $Z$. It follows that $\{\frac{Q}{l^2}, l\widetilde{Y}\}$ provides new minimal immersions. %together with  Consequently the vertices of the corresponding face are denoted by $\widetilde{Y}$. vertices %If $Z$ lies in the convex hull of $\{Y_j Y_j^t \mid Y_j\in Y\}$, then considering the star subdivision of the convex hull of $\{Y_j Y_j^t \mid Y_j\in Y\}$ with $\frac{Q^{-1}}{n}$ as the base point, and there must exist a subcell containing $Z$. then  decomposition of the convex hull of $\{Y_j Y_j^t \mid Y_j\in Y\}$ by convex polytopes with  up to a dilation, we may assume that $Z$ is integer and    

To show how many rational points lie on $\mathcal{Q}$, we write $Q=Q_1+\alpha Q_2$, where $Q_1, Q_2\in GL(n,\mathbb{Q})$ are two symmetric matrices and $[\mathbb{Q}(\alpha):\mathbb{Q}]=2$. Then a rational point $Z_1$ lies on $\mathcal{Q}$ if and only if 
$$Z_1^t Q_1 Z_1=1,~~~Z_1^t Q_2 Z_1=0,$$
 i.e., it lies on the intersection of two quadratic surfaces. In algebraic geometry, it is known (cf.   \cite[Chap. 4, Ex. 3.6]{Hartshorne}) that such an intersection is either a rational curve (possibly singular and reducible) or a quartic elliptic curve. In the rational case, the curve admits infinitely many rational points. In the elliptic case, the set of rational points forms a finitely generated abelian group (by the Mordell-Weil theorem, cf. \cite{Serre}), meaning it may be either finite (e.g., torsion points alone) or infinite (if the rank is positive). 
%we use an argument similar to that in Subsection~\ref{subsec-3tori}. 

\end{rem}
\iffalse{
\begin{rem}
    In terms of the matrix data $\{Y,Q\}$, a homogeneous minimal immersion $x$ is embedded if and only if the image of the map 
    \begin{equation*}
    \begin{split}
         \mathcal{Y}: \Omega\triangleq\{&(x_1,x_2,\cdots,x_n)\;|\; |x_i|<1\} \longrightarrow  \mathbb{R}^N \\
         &~~~~~~~(x_1,x_2,\cdots,x_n)~~~~~~~～\mapsto ~~~(x_1,x_2,\cdots,x_n)Y 
    \end{split}
    \end{equation*} %Y \cap \mathbb{Z}^N=\{0\}$$ 
    intersects $\mathbb{Z}^N$ solely at the origin. When $N=n$, by the first Minkowski's  Convex Body Theorem, this condition is equivalent to stating that $\det(Y)=\pm 1$ (this method has been mentioned in \cite{Park-Urakawa}). Similarly, when $N> n$, if one of the $n\times n$ minors of Y is equal to $\pm 1$, then $x$ is embedded. 
\end{rem}
}\fi

\subsection{Homogeneity of irrational minimal flat $3$-torus}
{It follows from %Proposition~\ref{prop-homo} and Theorem~\ref{prop-irrational} 
{Theorem~\ref{thm-3irra}} that %the uniqueness property in the last subsection 
the cubic and quartic irrational minimal flat $3$-tori in spheres are obviously homogeneous. In this subsection, we show that this homogeneity also holds in the quadratic case.} %The following lemma (see Lemma 2.3 in \cite{LWX}) will be used,} 
\begin{theorem}
    The minimal isometric immersions of irrational flat $3$-tori into spheres are homogeneous. %Let $T^3$ be an irrational flat torus, minimally immersed into some sphere, then the immersion is homogeneous.
\end{theorem}
\begin{proof}
     Assume that the matrix data associated to this minimal isometric immersion is $\{Q,Y\}$. { By the homogeneity proof in Theorem~\ref{thm-3irra},} 
     it needs only to prove for the case of $\rk\{Y_jY_j^t\,|\, Y_j\in Y\}=5$. 
    If there is a non-trivial $\eta$-set:
    \[
    Y_1+Y_2=Y_3+Y_4=Y_5+Y_6=\eta,
    \]
    then the plane $\{c_1Y_1+c_2Y_2+c_3Y_3|c_1+c_2+c_3=1\}$ intersects the ellipsoid $x^t Q x=1$ at an ellipse with center $\eta/2$. 
    
    Set $P\triangleq(\frac{\eta}{2},Y_1-\frac{\eta}{2},Y_3-\frac{\eta}{2})$, $\widetilde{Q}\triangleq P^tQP$, \text{and} $\widetilde{Y}_i\triangleq P^{-1}Y_i ~(1\leq i\leq 6)$. Then we have 
    \[\widetilde{Y}_i^t\widetilde{Q}\widetilde{Y}_i=1,\quad \widetilde{Y}_i=\begin{pmatrix}
        1\\ y_i
    \end{pmatrix},\]
    in which $y_i\in\mathbb{Q}^2$ for $1\leq i\leq 6$, and 
    \[
    y_1=-y_2=\begin{pmatrix}
        1\\0
    \end{pmatrix},\quad y_3=-y_4=\begin{pmatrix}
        0\\1
    \end{pmatrix},\quad y_5=-y_6.
    \]
    Note that the matrix $R\in \mathrm{Sym}_2$ solving $y_i^tRy_i=1$ for all $1\leq i\leq 6$ is uniquely determined and must be rational. Moreover, %the hyper-ellipsoid $x\overline{Q} x^t=1$ passing  we obtain that 
    there {exists} a number $a\in \mathbb{R}$ such that 
    \[
    \widetilde{Q}=\begin{pmatrix}
        a&\\&(1-a)R
    \end{pmatrix}.
    \]
    By assumption, $Q$ is irrational. It follows that $\widetilde{Q}$ is irrational, and hence $a$ is irrational. This implies that other $\widetilde{Y}_j (j>6)$ if existing, must be of the form $\begin{pmatrix}
        1\\y_j
    \end{pmatrix}$ with $y_j^tRy_j=1$. Consequently, we have 
    $$W_{\widetilde{Y}}=\left\{\begin{pmatrix}
        t&\\&(1-t)R
    \end{pmatrix} \,\bigg{|}\, t\in \mathbb{R}\right\}.$$

    By Theorem~\ref{thm-vari}, we derive that 
    %So the immersion being minimal implies that $t=a$ 
    $a$ achieves the maximum of determinant function $t(1-t)^2\mathrm{det}(R)$, which is {obviously} %can only be 
    rational. This contradiction implies that the non-trivial $\eta$-set consists of exactly $2$ pairs and thus the immersion is homogeneous by Lemma 2.3 in \cite{LWX}.
\end{proof}
\section{{Examples of non-homogeneous minimal flat $3$-tori in spheres}}\label{sec-defor}
{In this section, we show the existence of  non-homogeneous minimal isometric  immersions (even embeddings) of rational flat $3$-tori with $\rk\{{Y_iY_i^t}\mid 1\leq i\leq \sharp(Y)\}=5$. Moreover, by minimal products \cite{CH,Tang-Z,Xin} with the Clifford $(n-3)$-torus, one can show there exists non-homogeneous minimal flat $n$-tori in spheres for every $n\geq3$.} %being $5$.
From \eqref{eq-X}, it is easy to see that a minimal %immersion of 
flat $n$-tori is homogeneous {if and only if the matrix $A$ is diagonal up to multiplying an orthogonal matrix on the right, or equivalently,} the symmetric matrix $AA^t$ is diagonal. We begin by introducing a family of interesting $3$-tori.  

Let  
\[
\mathfrak{P}=\{\mathrm{p} \mid \mathrm{p}\ \text{is a prime number congruent to } 1 \text{ modulo } 4\}.
\]  
By Fermat's theorem on the sum of two squares \cite{Ireland-Rosen}, for every \(r \in \mathfrak{P}\) there exists a {unique} ordered pair of positive integers \((m_1,m_2)\) such that  
\[
r = m_1^2 + m_2^2,\qquad \gcd(m_1,m_2)=1,\qquad m_1 > m_2.
\]
For each given $r\in \mathfrak{P}$ and associated pair \((m_1,m_2)\),  %{define two integers $p, q$ such that $p<q$ and 
%$$\{p,q\}=\{m_1^2-m_2^2, 2 m_1 m_2\}.$$}
set  
$$p\triangleq m_1^2-m_2^2, ~~~~~~q\triangleq 2 m_1 m_2.$$
Then $\{p,q,r\}$ forms a primitive Pythagorean triple, i.e., they satisfy 
    $$p^2+q^2=r^2,~~~%0<p<q<r,~~~
    (p,q,r)=1.$$
% \begin{lemma}
%     $r\in\mathfrak{P}$ if and only if the associated pair \((m_1,m_2)\) is unique, or equivalently, $r$ is not the hypotenuse of some primitive Pythagorean triple except $\{p,q,r\}$.
% \end{lemma}
% \begin{proof}
%     By Fermat’s theorem on sums of two squares, for any $\mathrm{p}$ there is a unique $a+\mathbf{i}b\in\mathbb{Z}[\mathbf{i}]$ such that
%     \[
%     \mathrm{p}=a^2+b^2=(a+\mathbf{i}b)\overline{(a+\mathbf{i}b)},\quad a>b>0.
%     \]
%     So
%     \[
%     \mathrm{p}^2=(a+\mathbf{i}b)^2\overline{(a+\mathbf{i}b)^2}.
%     \]
%     Then we get $(m_1,m_2)=(a,b)$ for $\mathrm{p}$; $(m_1,m_2)=(a^2-b^2,2ab)$ or $(2ab,a^2-b^2)$ for $\mathrm{p}^2$. The choice of $(m_1,m_2)$ is obviously unique.

    Conversely, given a primitive Pythagorean triple $\{p,q,r\}$, we claim that $r\in \mathfrak{P}$ provided that it is not the hypotenuse of any other Pythagorean triple. Note that $r$ is odd and it is also a sum of two squares. Suppose that $r$ %(must be odd) 
    admits a prime factorization
    \[
    r=\prod_{i\in\Gamma}(4k_i+1)^{f_i}\cdot\prod_{i\in\Gamma'}(4k_i+3)^{g_i}.
    \]
Since any {primitive} divisor of a sum of two squares is itself a sum of two squares, we obtain \(g_i = 0\) for all \(i \in \Gamma'\) because a number of the form \(4k_i + 3\) can never be expressed as a sum of two nonzero squares. By assumption, we have \(4k_1 + 1 \in \mathfrak{P}\); hence there exists a unique ordered pair \((a_1, b_1)\) of positive integers satisfying \(4k_1 + 1 = a_1^2 + b_1^2\).  %$a_1$ and $b_1$ 
Set $r'=r/(4k_1+1)\in\mathbb{Z}$, 
   % \[ 4k_1+1=\|a_1+\mathbf{i}b_1\|^2\in\mathfrak{P},\quad r'=r/(4k_1+1)\in\mathbb{Z},\quad a_1>b_1>0,\]
then one can verify that 
    \[
    \{r'(a_1^2-b_1^2),r'(2a_1b_1),r\}%r'(a_1^2+b_1^2)\},
    \]
    is also a Pythagorean triple with $r$ as 
    the hypotenuse. This is a new Pythagorean triple  if $r'\neq1$. So we show the claim holds.  
%     \begin{align*}
%     r&=\|(a_1+\mathbf{i}b_1)(a_2+\mathbf{i}b_2)\|^2=\|(a_1a_2-b_1b_2)+\mathbf{i}(a_2b_1+a_1b_2)\|^2\\
%     &=\|(a_1+\mathbf{i}b_1)(a_2-\mathbf{i}b_2)\|^2=\|(a_1a_2+b_1b_2)+\mathbf{i}(a_2b_1-a_1b_2)\|^2.
%     \end{align*}
%     It is easy to see that
%     \[
%     a_1a_2+b_1b_2\not=a_1a_2-b_1b_2,\quad (a_1a_2+b_1b_2)-(a_2b_1+a_1b_2)=(a_1-b_1)(a_2-b_2)\not=0
%     \]
%     which means $r$ is at least the hypotenuse of two distinct Pythagorean triples
% \end{proof}

% Conversely, given a primitive Pythagorean triple $\{p,q,r\}$, if we require that $r$ can not be the hypotenuse of another primitive Pythagorean triple, then one can show that $r$ must belong to $\mathfrak{P}$. This is because 
% \[
% m_1=\sqrt{\frac{r+p}{2}},\quad m_2=\sqrt{\frac{r-p}{2}}.
% \]

    %Let $\{p,q,r\}$ be a primitive Pythagorean triple, i.e., they satisfy 
    %$$p^2+q^2=r^2,~~~0<p<q<r,~~~(p,q,r)=1.$$
    %In addition, we require that $r$ can not be the hypothenuse of another primitive Pythagorean triple. Under this assumption, a primitive Pythagorean triple is uniquely determined by the parameter $r$, and $r$ can be classified as either a prime number congruent to $1$ modulo $4$, or the square of such a prime. 
    %by choosing $r$ to be a prime numbers modulo $4$ with remainder $1$ we can obtain infinitely many such primitive Pythagorean triples. 
    %$p,q,r\in\mathbb{N}$ be three co-prime integers satisfying $p^2+q^2=r^2$ and $p<q$ are uniquely determined. 
    %$${e^{iu_1}\left(e^{iu_2}, e^{iu_3}, e^{-iu_2}, e^{-iu_3}\right)A}$$

   Given $r\in \mathfrak{P}$, consider the lattice $\Lambda_{3,r}^*$ generated by  
\begin{equation}\label{eq-generator}
\eta_1=\begin{pmatrix}  2/\sqrt{3}\\0\\0\end{pmatrix},\quad
\eta_2=\frac{\sqrt{2}}{\sqrt{3}}\begin{pmatrix}  -1/\sqrt{2}\\1/r\\0\end{pmatrix},\quad
\eta_3=\frac{\sqrt{2}}{\sqrt{3}}\begin{pmatrix}  -1/\sqrt{2}\\0\\1/r\end{pmatrix}.
\end{equation}
We denote by $\mathbb{T}^3_r=\mathbb{R}^3/\Lambda_{3,r}$ the flat $3$-torus determined by its dual lattice $\Lambda_{3,r}$.

\subsection{A family of non-homogeneous minimal isometric immersions of $\mathbb{T}^3_r$ into spheres}
 %{According to Proposition \ref{prop-homo}, a strategy to find non-homogeneous $\lambda_1$-minimal flat tori might be the following: one can first classify the homogeneous $\lambda_1$-minimal flat tori, so all the possible sets $\{Y_i\}$ are obtained; then check all the $\eta$-sets to see if \eqref{eq-eigen1} and \eqref{eq-iso1} admit non-trivial solutions.}
%\subsection{Non-homogeneous examples}
Let $\mathbb{T}^3_r$ be a flat $3$-torus determined by an integer $r\in \mathfrak{P}$. By the uniqueness of the primitive Pythagorean triple determined by $r$, one can verify directly  that %Check the following $\xi_i$:
\[
\xi_1=\frac{\sqrt{2}}{\sqrt{3}}\begin{pmatrix}1/\sqrt{2}\\1\\0\end{pmatrix},\quad\xi_2=\frac{\sqrt{2}}{\sqrt{3}}\begin{pmatrix}1/\sqrt{2}\\-1\\0\end{pmatrix},\quad\xi_3=\frac{\sqrt{2}}{\sqrt{3}}\begin{pmatrix}1/\sqrt{2}\\0\\1\end{pmatrix},\quad\xi_4=\frac{\sqrt{2}}{\sqrt{3}}\begin{pmatrix}1/\sqrt{2}\\0\\-1\end{pmatrix},
\]
\[
\xi_5=\frac{\sqrt{2}}{\sqrt{3}}\begin{pmatrix}1/\sqrt{2}\\p/r\\q/r\end{pmatrix},\quad\xi_6=\frac{\sqrt{2}}{\sqrt{3}}\begin{pmatrix}1/\sqrt{2}\\-p/r\\-q/r\end{pmatrix},\quad\xi_7=\frac{\sqrt{2}}{\sqrt{3}}\begin{pmatrix}1/\sqrt{2}\\-q/r\\p/r\end{pmatrix},\quad\xi_8=\frac{\sqrt{2}}{\sqrt{3}}\begin{pmatrix}1/\sqrt{2}\\q/r\\-p/r\end{pmatrix},
\]
\[
\xi_9=\frac{\sqrt{2}}{\sqrt{3}}\begin{pmatrix}1/\sqrt{2}\\p/r\\-q/r\end{pmatrix},\quad\xi_{10}=\frac{\sqrt{2}}{\sqrt{3}}\begin{pmatrix}1/\sqrt{2}\\-p/r\\q/r\end{pmatrix},\quad\xi_{11}=\frac{\sqrt{2}}{\sqrt{3}}\begin{pmatrix}1/\sqrt{2}\\q/r\\p/r\end{pmatrix},\quad\xi_{12}=\frac{\sqrt{2}}{\sqrt{3}}\begin{pmatrix}1/\sqrt{2}\\-q/r\\-p/r\end{pmatrix}
\]
%which 
constitute all the vectors (up to $\pm 1$) of norm %$\sqrt{3}/\sqrt{2}$
$1$ in the lattice $\Lambda_{3,r}^*$.  

To construct the minimal isometric immersions of $\mathbb{T}^3_r$ into spheres, particularly those that are non-homogeneous, we first analyze $\{\xi_i\}_{j=1}^{12}$ to identify all $\eta$-sets (see Definition~\ref{de-etaset}). Then, by solving equations \eqref{eq-eigen1}, \eqref{eq-eigen2}, \eqref{eq-iso1}, and \eqref{eq-iso2} for these $\eta$-sets, together with \eqref{eq-euta}, we obtain the expression for $AA^t$. It will be shown that $AA^t$ can be non-diagonal, implying that the minimal immersion of the form \eqref{eq-X} is non-homogeneous.

%To construct the example, we will study $\{\xi_i\}$ first to determine all the $\eta$-sets; then by solving equations \eqref{eq-eigen1}, \eqref{eq-eigen2}, \eqref{eq-iso1} and \eqref{eq-iso2} for those $\eta$-sets, we obtain the expression of $AA^t$. It will be shown that $AA^t$ can be non-diagonal, so that the minimal immersion as in \eqref{eq-X} is non-homogeneous.

Define $\eta_0=(\frac{2}{\sqrt{3}},0,0)^t$. The following $\eta_0$-set is obvious: 
\[
\eta_0=\xi_1+\xi_2=\xi_3+\xi_4=\cdots=\xi_{11}+\xi_{12}.
\]
In order to find all the possible $\eta$-sets, we define a projection $P: \mathbb{R}^3 \rightarrow \mathbb{R}^2$ by
\[
P v=\begin{pmatrix}
    0&\frac{\sqrt{3}}{\sqrt{2}}&0\\0&0&\frac{\sqrt{3}}{\sqrt{2}}
\end{pmatrix}v,\quad v\in \mathbb{R}^3.
\]
So $\{P\xi_i,1\leq i\leq 12\}$ consists of all the unit normal vectors of lattice  generated by $(\frac{1}{r},0)^t$ and $(0,\frac{1}{r})^t$. 
\begin{lemma}
 The aforementioned $\eta_0$-set is the only non-trivial $\eta$-set.   
\end{lemma}
%Next we will show the aforementioned $\eta_0$-set is the only non-trivial $\eta$-set. 
This lemma will be established through the following two lemmas. Note that for any given $i<j$ and $k<l$, the pairs $(\xi_i, \xi_j)$ and $(\xi_k, -\xi_l)$ cannot belong to the same $\eta$-set. 
\begin{lemma}\label{lem-a+b}
    If $(a,b)\not=(2k-1,2k)$ for some $k$ then the  $(\xi_a+\xi_b)$-set contains only $(\xi_a,\xi_b)$.
\end{lemma}
\begin{proof}
    %Obviously, $v\triangleq P(\xi_a+\xi_b)=P\xi_a+P\xi_b\not=0$. Since $P\xi_a$ and $P\xi_b$ are both unit normal, they must correspond to the intersection of unit circle centered at origin $o$ and the perpendicular bisector of the line segment $ov$. $\frac{1}{2}v+tv^\perp$ $(t\in\mathbb{R})$, in which $v^\perp$ is a vector perpendicular to $v$. 
    Obviously, $v \triangleq P(\xi_a + \xi_b) = P\xi_a + P\xi_b \neq 0$. %The conclusion follows from that 
    For any pair of unit normal vectors whose sum equals $v$, they  %Since $P\xi_a$ and $P\xi_b$ are both unit normal vectors, they 
    must correspond to the two intersection points of the unit circle centered at the origin $o$ and the perpendicular bisector of the line segment $ov$. 
    The conclusion comes from the uniqueness of such intersection.
\end{proof}
\begin{lemma}\label{lem-a-b}
    For all $1\leq a<b\leq 12$, the $(\xi_a-\xi_b)$-set contains only $(\xi_a,-\xi_b)$.
\end{lemma}
\begin{proof}
Suppose there exist %s another pair  %
two distinct pairs $(\xi_a, -\xi_b)$ and 
$(\xi_c, -\xi_d)$ in %the 
an $\eta$-set. By Definition~\ref{de-etaset}, we have %such that 
$$a<b,~~~c<d,~~~\xi_c-\xi_d=\xi_a-\xi_b,$$
which yields that $\xi_a+\xi_d=\xi_c+\xi_b$. Without loss of generality, assume $a<c$. Then, from Lemma \ref{lem-a+b} we obtain a contradiction that $a<c<d=a+1$. 
%$a=c$ and $b=d$, or 
%$$d=a+1,~~~b=c\pm 1,$$ 
%which gives a contradiction.  
%, by the definition of $\eta$-set (see Definition~\ref{de-etaset}) we get $c<d$ which yields a contradiction, $a<c<d=a+1$.
\end{proof}

Considering $(\xi_a+\xi_b)$-set and $(\xi_a-\xi_b)$-set for $(a,b)\not=(2k-1,2k)$, we get from \eqref{eq-eigen1} and \eqref{eq-eigen2} that the block of $AA^t$ (see \eqref{eq-AAt}) 
\begin{equation}\label{eq-Aab}
A_{ab}=O,
\end{equation}
which obviously satisfies the linear system \eqref{eq-iso1} and \eqref{eq-iso2} for $r=a,s=b$. 
Similarly, considering $(\xi_{2k-1}-\xi_{2k})$-set, we derive that 
\begin{equation}\label{eq-Akk}
A_{ij}^{11}+A_{ij}^{22}=A_{ij}^{12}-A_{ij}^{21}=0,\quad  i=2k-1,~j=2k, ~1\leq k\leq 6.
\end{equation}

Combining \eqref{eq-Arr}, \eqref{eq-Aab} and \eqref{eq-Akk}, we obtain that $AA^t$ has the following block structure:   
\[
AA^t=\begin{pmatrix}
    B_1&&&\\&B_2&&\\&&\ddots&\\&&&B_6
\end{pmatrix},\quad 
B_i=\begin{pmatrix}
    a_{2i-1}I_2& A_{2i-1\,2i}\\A_{2i-1\,2i}^t&a_{2i}I_2
\end{pmatrix},\quad
A_{2i-1\,2i}=\begin{pmatrix}
    \alpha_i&\beta_i\\\beta_i&-\alpha_i
\end{pmatrix}.
\]

We now turn our attention to solving equations \eqref{eq-iso1} and \eqref{eq-iso2} for the $\eta_0$-set specified above, noting that these equations incorporate \eqref{eq-eigen1} and \eqref{eq-eigen2} as special cases. A direct verification shows that the solutions are given by %Equation \eqref{eq-iso1} for $\eta_0$-set yields
\begin{equation}\label{eq-ab}
\left\{
\begin{aligned}
\alpha_1&=-R_1\left(1+\frac{p^2-q^2}{r^2}(\cos^2\phi_1-\cos^2\psi_1)\right),&&\alpha_3=R_1\cos^2\phi_1,&&\alpha_5=R_1\sin^2\psi_1,\\
\alpha_2&=-R_1\left(1+\frac{p^2-q^2}{r^2}(\cos^2\psi_1-\cos^2\phi_1)\right),&&\alpha_4=R_1\cos^2\psi_1,&&\alpha_6=R_1\sin^2\phi_1,\\
\beta_1&=-R_2\left(1+\frac{p^2-q^2}{r^2}(\cos^2\phi_2-\cos^2\psi_2)\right),&&\beta_3=R_2\cos^2\phi_2,&&\beta_5=R_2\sin^2\psi_2,\\
\beta_2&=-R_2\left(1+\frac{p^2-q^2}{r^2}(\cos^2\psi_2-\cos^2\phi_2)\right),&&\beta_4=R_2\cos^2\psi_2,&& \beta_6=R_2\sin^2\phi_2,
\end{aligned}
\right.
\end{equation}
in which $R_1, R_2, \phi_1, \phi_2, \psi_1$ and $\psi_2$ are constant parameters. 
\iffalse{Similarly, 
equation \eqref{eq-iso2} yields
\[
\beta_1=-R_2\left(1+\frac{p^2-q^2}{r^2}(\cos^2\phi_2-\cos^2\psi_2)\right),\quad
\beta_2=-R_2\left(1+\frac{p^2-q^2}{r^2}(\cos^2\psi_2-\cos^2\phi_2)\right),
\]
\[
\beta_3=R_2\cos^2\phi_2,\quad \beta_4=R_2\cos^2\psi_2,\quad \beta_5=R_2\sin^2\psi_2,\quad \beta_6=R_2\sin^2\phi_2.
\]}\fi
% A_{2i-1\,2i}

By a long but straightforward computation, the equation \eqref{eq-euta} can be transformed to the following linear system, 
\begin{equation}\label{eq-ai}
\left\{
\begin{aligned}
a_5+a_6+a_{11}+a_{12}&=a_7+a_8+a_9+a_{10},\\
a_1+\dfrac{p(p+r)}{2r^2}(a_5+a_9)+\dfrac{p(p-r)}{2r^2}(a_6+a_{10})&+\dfrac{q(q+r)}{2r^2}(a_8+a_{11})+\dfrac{q(q-r)}{2r^2}(a_7+a_{12})=\dfrac{1}{4},\\
a_3+\dfrac{q(q+r)}{2r^2}(a_5+a_{10})+\dfrac{q(q-r)}{2r^2}(a_6+a_9)&+\dfrac{p(p+r)}{2r^2}(a_7+a_{11})+\dfrac{p(p-r)}{2r^2}(a_8+a_{12})=\dfrac{1}{4},\\
a_2+\dfrac{p(p-r)}{2r^2}(a_5+a_{9})+\dfrac{p(p+r)}{2r^2}(a_6+a_{10})&+\dfrac{q(q-r)}{2r^2}(a_8+a_{11})+\dfrac{q(q+r)}{2r^2}(a_7+a_{12})=\dfrac{1}{4},\\
a_4+\dfrac{q(q-r)}{2r^2}(a_5+a_{10})+\dfrac{q(q+r)}{2r^2}(a_6+a_9)&+\dfrac{p(p-r)}{2r^2}(a_7+a_{11})+\dfrac{p(p+r)}{2r^2}(a_8+a_{12})=\dfrac{1}{4}.
\end{aligned}
\right.
\end{equation}
Clearly, 
 $a_1,\cdots,a_4$ can be expressed in terms of the remaining eight parameters. Moreover, the first equation shows that the solution space of \eqref{eq-ai} is 7-dimensional. 

%Moreover, 
Finally, to obtain a minimal isometric immersion, we need only further impose that $AA^t$ is positive semidefinite, or equivalently, 
each block $B_i$ must be positive semidefinite. This is equivalent to requiring that \( \alpha_i^2 + \beta_i^2 \), the eigenvalue of \( A_{2i-1\,2i}^t A_{2i-1\,2i} \), does not exceed \( a_{2i-1}a_{2i} \), a condition that can be satisfied by choosing \( R_1 \) and \( R_2 \) sufficiently small. 

In summary, once the entries \(a_i, \alpha_j, \beta_j\) of \(AA^t\) have been determined, we take the square root of \(AA^t\) to obtain the required matrix \(A\) in \eqref{eq-X}, which then yields a minimal isometric immersion of \(\mathbb{T}^3_r\) into \(\mathbb{S}^{23}\). As long as one of $\{\alpha_i,\beta_i \mid 1\leq i\leq 6\}$ is nonzero, the corresponding minimal immersion is non-homogeneous. Since each parameter $\alpha_i$ and $\beta_i$ can be deformed continuously to $0$, every such minimal isometric immersion can be deformed  continuously into a homogeneous example, which itself forms a $7$-parameter family. Moreover, any  homogeneous case can be further deformed into $a_{2i-1}=a_{2i}=0$ for all $i>2$, resulting in a homogeneous minimal immersion in $\mathbb{S}^7$. Thus we arrive at the following proposition. %, we get a family of non-congruent, minimal isometric immersions of flat torus $\mathrm{T}^3=\mathbb{R}^3/\Lambda_3$ into $\mathbb{S}^{23}$, %as in \eqref{eq-X}, 
%where $\Lambda_3$ is the dual lattice of $\Lambda_3^*$. %They are not congruent to each other. 
\begin{proposition}
For each $r\in \mathfrak{P}$, the flat $3$-torus $\mathbb{T}^3_r$ admits a $13$-parameter family of non-congruent minimal isometric immersions into spheres. Apart from a $7$-dimensional subset, %exceptional set in the parameter space, 
every minimal immersion of $\mathbb{T}^3_r$ in this family is non-homogeneous. %ones constitute an open and dense set. 
\end{proposition}

%The parameter space of such immersions has dimension $7+6=13$. 

%The matrix data $\{Q,Y\}$ of such 
%immersions is given by 
Next, we show the matrix data associated with such minimal immersions and use it to analyze their embeddedness. Define 
\[
Q=\frac{1}{3}\left(
\begin{array}{ccc}
 4 & -2 & -2 \\
 -2 & \frac{r^2+2}{r^2} & 1 \\
 -2 & 1 & \frac{r^2+2}{r^2} \\
\end{array}
\right),\]
\begin{equation}\label{eq-Y}
Y=\left(
\begin{array}{cccccccccccc}
 \frac{r+1}{2} & \frac{1-r}{2} & \frac{r+1}{2} & \frac{1-r}{2} & \frac{1+p+q}{2} & \frac{1-p-q}{2} & \frac{1+p-q}{2}& \frac{1-p+q}{2} & \frac{1+p-q}{2}  & \frac{1-p+q}{2} & \frac{1+p+q}{2}  & \frac{1-p-q}{2} \\
 r & -r & 0 & 0 & p & -p & -q & q & p & -p & q & -q \\
 0 & 0 & r & -r & q & -q & p & -p & -q & q & p & -p \\
\end{array}
\right).    
\end{equation}
\iffalse{\[
Q=\frac{1}{3}\begin{pmatrix}
    1&&\\&\frac{2}{r^2}&\\&&\frac{2}{r^2}
\end{pmatrix},\quad
Y=\left(\begin{array}{cccccccccccc}
    1&1&1&1&1&1&1&1&1&1&1&1\\r&-r&0&0&p&-p&-q&q&p&-p&q&-q\\0&0&r&-r&q&-q&p&-p&-q&q&p&-p
\end{array}\right).
\]
}\fi
%determining a family of minimal isometric immersions with $7+6=13$ parameters. Since all the parameters $\alpha_i$ and $\beta_i$ can deform continuously to $0$, each of these minimal isometric immersions can deform into homogeneous cases, which is still a family with $7$ parameters. Furthermore, any of these homogeneous cases can deform into $a_{2i}=a_{2i+1}=0$ for all $i>2$, ending up with a minimal immersion in $\mathbb{S}^7$. %and $r^2$-covered by the given torus.
%\begin{rem}
%Let $\{e_i=(e_{i1},e_{i2},e_{i3})^t\}_{i=1,2,3}$ be the generator of $\Lambda_3$, $\{e_i^*=(e^*_{i1},e^*_{i2},e^*_{i3})^t\}_{i=1,2,3}$  the corresponding generator of $\Lambda_3^*$. So $e_i^te^*_j=\delta_{ij}$. For any $u=(e_1,e_2,e_3)\mathbf{u}$, the $\theta_i$ defined in \eqref{eq-X} has an explicit expression 
\iffalse
\[
\theta_i=2\pi\langle\xi_i,u\rangle=2\pi\langle (e_1^*,e_2^*,e_3^*)Y_i,(e_1,e_2,e_3)\mathbf{u}\rangle=2\pi\langle Y_i,\mathbf{u}\rangle.
\]
\fi 
%It is well known that for any given primitive Pythagorean triple $\{p,q,r\}$, 
Note that, by definition, $r$ and $p$ are odd while $q$ is even. Hence the columns of $Y$ are all 
integral vectors. Moreover, one can verify directly that $\xi_j$ is obtained as the linear combination of $\eta_1,\eta_2,\eta_3$ (given in \eqref{eq-generator}) with coefficients given by the $j$-th column vector $Y_j$. Consequently, for a minimal immersion constructed above that uses all vectors $\{\xi_j\}_{j=1}^{12}$, the pair $\{Q, Y\}$ can be regarded as its matrix data.  
\begin{proposition}\label{non-homoge}
    For each $r\in \mathfrak{P}$, every full minimal isometric immersion of the flat $3$-torus $\mathbb{T}^3_r$ into $\mathbb{S}^{23}$ is embedded.
\end{proposition}
%\vskip 0.3cm
%{\bf Claim 1.} {\em Every full minimal isometric immersion of the flat $3$-torus $\mathrm{T}^3$ into $\mathbb{S}^{23}$ is embedded.}
%\vskip 0.2cm 
\begin{proof}
By the fullness assumption, the pair 
$\{Q,Y\}$ provides the matrix data of such an immersion. One can see that in \eqref{eq-em}, if $\mathbf{u}Y$ is a nonzero integer vector with its coordinates satisfy $0\leq u_j<1$, then  %$\mathbf{u}\in \Omega$, then from %=\mathbf{0}$ then 
from $\mathbf{u}(Y_1+Y_2)\in\mathbb{Z}$ we have $\mathrm{u}_1=0$. By 
$$\mathbf{u}Y_1\in \mathbb{Z},~~~\mathbf{u}(Y_5+Y_9)\in \mathbb{Z},~~~\mathbf{u}(Y_7+Y_{12})\in \mathbb{Z},$$
we derive 
$$r \mathrm{u}_2\in \mathbb{Z},~~~2p \mathrm{u}_2\in \mathbb{Z},~~~2q \mathrm{u}_2\in \mathbb{Z},$$
which yields that $\mathrm{u}_2=0$, since $(p,q,r)=1$ and $r$ is odd. Similarly, there holds $\mathrm{u}_3=0$. 
%(resp. $\mathbf{u}Y_3\in \mathbb{Z}$) it follows that $u_2=0$ (resp. $u_3=0$). 
Hence $\mathbf{u} = \mathbf{0}$. By Remark~\ref{rk-em}, every homogeneous minimal isometric immersion of $\mathbb{T}^3_r$ associated to the matrix data $\{Q, Y\}$ is an embedding. In fact, the analysis in Remark~\ref{rk-em} extends to non-homogeneous full immersions as well, since in this setting the matrix $A$ in \eqref{eq-X} has full rank. This completes the proof of this proposition. 
\end{proof}
%every full minimal isometric immersion of this $T^3$ into $\mathbb{S}^{23}$ is embedded, because in this case the matrix $A$ has full rank. 
\begin{rem}
    %Since the immersion is not necessary to use all the eigenvectors, we can construct examples as above with exactly one Pythagorean triple, even $r$ is the hypotenuse of multiple Pythagorean triples. 
    Because the minimal isometric immersion does not require the use of all eigenvectors, we can construct examples of the type described above even when \(r\) is the hypotenuse of multiple Pythagorean triples. Consequently, for any integer that can be expressed as a sum of two distinct nonzero squares, one can associate a flat \(3\)-torus together with non-homogeneous minimal isometric immersions (embeddings) into spheres. 
\end{rem}

\begin{rem}
By Corollary~\ref{coro-em}, for flat 
$2$-tori, the image of every minimal isometric immersion into a sphere is embedded. For higher-dimensional flat tori, Proposition~\ref{prop-em} demonstrates that the same conclusion remains valid for any homogeneous minimal isometric immersions. Although Proposition~\ref{non-homoge} establishes the existence of abundant non-homogeneous minimal isometric embeddings of flat $3$-tori into spheres, we shall present an example in the following subsection to illustrate that, even in dimension three, the image of a non-homogeneous minimal isometric immersion is not always embedded. 
%we will provide an example in the next subsection to illustrate that embeddedness cannot always be expected for the image of  non-homogeneous minimal isometric immersions of flat $n$-tori with $n\geq3$ into spheres.  
\end{rem}
\subsection{A minimal flat $3$-torus in $\mathbb{S}^7$ whose image exhibits self-intersections}

%A natural question arises among the minimal isometric immersions constructed above: do there exist non‑homogeneous embedded examples with the smallest possible codimension, i.e., in $\mathbb{S}^7$? If not, are the corresponding images at least embedded? 
We continue to adopt the notation introduced in the previous subsection. 
\begin{example}\label{ex-nonS7}
    In \eqref{eq-ab} and \eqref{eq-ai}, choose the following solution,    
$$\alpha_1=\alpha_2=-\alpha_3=-\alpha_4=\frac{1}{8},~~~\alpha_5=\alpha_6=0,~~~\beta_1=\cdots=\beta_6=0,$$
$$a_1=\cdots=a_8=\frac{1}{8},~~~a_9=\cdots=a_{12}=0.$$
Then we can obtain a matrix 
\[
A=\frac{1}{\sqrt{8}}\begin{pmatrix}
    A_1&&&&\\&\ddots&&&\\&&A_4&&\\&&&O&\\&&&&O
\end{pmatrix},\quad 
A_1=A_2=\begin{pmatrix}
    1& 0&0&0\\0&1&0&0\\1&0&0&0\\0&-1&0&0
\end{pmatrix},\quad
A_3=A_4=\begin{pmatrix}
    1&0 &0&0\\0&1&0&0\\-1&0&0&0\\0&1&0&0
\end{pmatrix},
\]
which yields a non-homogeneous minimal isometric immersion of $\mathrm{x}_r: \mathbb{T}^3_r \rightarrow \mathbb{S}^7\subset\mathbb{S}^{15}$ by \eqref{eq-X}. 
\end{example}
We first show that this immersion is not an embedding. Set \[
\widetilde Q=\left(
\begin{array}{ccc}
 1 & 0 & 0 \\
 0 & m_1 & m_2 \\
 0 & -m_2 & m_1 \\
\end{array}
\right)Q\left(
\begin{array}{ccc}
 1 & 0 & 0 \\
 0 & m_1 & -m_2 \\
 0 & m_2 & m_1 \\
\end{array}
\right),\]
%\frac{1}{3}\left(
%\begin{array}{ccc}
% 4 & -2 (m_1+m_2) & -2(m_1-m_2) \\
% -2 (m_1+m_2) & \frac{2}{m_1^2+m_2^2}+(m_1+m_2)^2 &  
% m_1^2-m_2^2 \\
% -2(m_1-m_2) & m_1^2-m_2^2 & \frac{2}{m_1^2+m_2^2}+%(m_1-m_2)^2 \\
%\end{array}
%\right),\]
\begin{equation}\label{eq-TY}
\widetilde Y=\left(
\begin{array}{cccccccccccc}
 \frac{r+1}{2} & \frac{1-r}{2} & \frac{r+1}{2} & \frac{1-r}{2} & \frac{1+p+q}{2} & \frac{1-p-q}{2} & \frac{1+p-q}{2}& \frac{1-p+q}{2} \\
 m_1 & -m_1 & m_2 & -m_2 & m_1 & -m_1 & -m_2 & m_2  \\
 -m_2 & m_2 & m_1 & -m_1 & m_2 & -m_2 & m_1 & -m_1 \\
\end{array}
\right),
\end{equation}
where $\{m_1,m_2\}$ is the unique ordered pair of coprime positive integers satisfying $r=m_1^2+m_2^2$. 
%$$p=u^2-v^2,~~~q=2uv,~~~r=u^2+v^2,~~~(u,v)=1,~~~u>v.$$
Let $\widetilde{\Lambda}_{3,r}^*$ be the lattice generated by 
\begin{equation}\widetilde{\eta}_1\triangleq\eta_1,~~~ \widetilde{\eta}_2\triangleq m_1\eta_2+m_2\eta_3,~~~ \widetilde{\eta}_3\triangleq m_1\eta_3-m_2\eta_2.
\end{equation}
It takes $\widetilde{Q}$ as the Gram matrix.  We denote by $\widetilde{\Lambda}_{3,r}$ its dual lattice, and 
$\widetilde{\mathbb{T}}^3_r=\mathbb{R}^3/\widetilde{\Lambda}_{3,r}$ the corresponding flat $3$-torus.  %determined by this lattice. 
Note that $\mathbb{T}^3_r$ forms a $r$-fold cover of $\widetilde{\mathbb{T}}^3_r$.  
For $1\leq j\leq 8$, one can verify directly that $\xi_j$ is exactly the linear combination of $\widetilde{\eta}_1, \widetilde{\eta}_2, \widetilde{\eta}_3$ with coefficients given by the $j$-th column vector $\widetilde{Y}_j$. Consequently, the  minimal isometric immersion $\mathrm{x}_r: \mathbb{T}^3_r \rightarrow \mathbb{S}^7$ factors through a minimal isometric immersion $\widetilde{\mathrm{x}}_r: \widetilde{\mathbb{T}}^3_r \rightarrow \mathbb{S}^7$, whose matrix data is $\{\widetilde Q, \widetilde Y\}$. By \eqref{eq-X}, \eqref{eq-em} and \eqref{eq-TY}, the immersion $\widetilde{\mathrm{x}}_r$ admits the following explicit expression: 
%$$\widetilde{\mathrm{x}}_r(\mathrm{u}_1,\mathrm{u}_2,\mathrm{u}_3)=$$
\begin{equation}\label{eq-example}
\begin{split}
\widetilde{\mathrm{x}}_r(\mathrm{u}_1,\mathrm{u}_2,\mathrm{u}_3) 
= \frac{1}{\sqrt{2}} \Bigl(
&\cos(\mathrm{u}_1\pi)\, e^{2\pi \mathbf{i}\bigl(\frac{r}{2}\mathrm{u}_1 + m_1 \mathrm{u}_2 - m_2 \mathrm{u}_3\bigr)},\ 
\cos(\mathrm{u}_1\pi)\, e^{2\pi \mathbf{i}\bigl(\frac{r}{2}\mathrm{u}_1 + m_2 \mathrm{u}_2 + m_1 \mathrm{u}_3\bigr)},\\
&\sin(\mathrm{u}_1\pi)\, e^{2\pi \mathbf{i}\bigl(\frac{p+q}{2}\mathrm{u}_1 + m_1 \mathrm{u}_2 + m_2 \mathrm{u}_3\bigr)},\ 
\sin(\mathrm{u}_1\pi)\, e^{2\pi \mathbf{i}\bigl(\frac{p-q}{2}\mathrm{u}_1 - m_2 \mathrm{u}_2 + m_1 \mathrm{u}_3\bigr)} \Bigr),
\end{split}
\end{equation}
with $0\leq \mathrm{u}_1,\mathrm{u}_2,\mathrm{u}_3 <1$. 

Next, we show that even $\widetilde{\mathrm{x}}_r: \widetilde{\mathbb{T}}^3_r \rightarrow \mathbb{S}^7$ fails to be an embedding, and its image %{\color{red}its image} 
possesses self-intersections. Note that for any $\mathrm{u}_1,\mathrm{u}_1'\in (0,\frac{1}{2})\cup(\frac{1}{2},1)$, $\widetilde{\mathrm{x}}_r(\mathrm{u}_1,\mathrm{u}_2,\mathrm{u}_3)=\widetilde{\mathrm{x}}_r(\mathrm{u}_1',\mathrm{u}_2',\mathrm{u}_3')$ implies $\mathrm{u}_1=\mathrm{u}_1'$ or $\mathrm{u}_1=1-\mathrm{u}_1'$. Set $\delta_2\triangleq \mathrm{u}_2-\mathrm{u}_2'$, $\delta_3\triangleq \mathrm{u}_3-\mathrm{u}_3'$. Then in the former case, we get 
$$ m_1\delta_2-m_2\delta_3\in\mathbb{Z},~~
    m_2\delta_2+m_1\delta_3\in\mathbb{Z},~~
    m_1\delta_2+m_2\delta_3\in\mathbb{Z},~~
    -m_2\delta_2+m_1\delta_3\in\mathbb{Z}.$$
Then we have
\begin{equation}\label{eq-rp}
\begin{aligned}
    r\delta_2&=m_1(m_1\delta_2-m_2\delta_3)+m_2(m_2\delta_2+m_1\delta_3)\in\mathbb{Z},\\
    2p\delta_2&=(m_1-m_2)[(m_1\delta_2-m_2\delta_3)+(m_2\delta_2+m_1\delta_3)+(m_1\delta_2+m_2\delta_3)+(m_2\delta_2-m_1\delta_3)]\in\mathbb{Z}.
\end{aligned}
\end{equation}
Since $(r,p)=1$, this forces $\delta_2\in\mathbb{Z}$, and consequently  $\delta_2=0$. It follows that $m_1\delta_3\in\mathbb{Z}$ and $m_2\delta_3\in\mathbb{Z}$, which further implies $\delta_3=0$ since $(m_1,m_2)=1$.  
In the latter case, we obtain 
\[\begin{gathered}\frac{1}{2}+\frac{r}{2}(2\mathrm{u}_1-1)+m_1\delta_2-m_2\delta_3\in\mathbb{Z},~~
    \frac{1}{2}+\frac{r}{2}(2\mathrm{u}_1-1)+m_2\delta_2+m_1\delta_3\in\mathbb{Z},\\
   \frac{p+q}{2}(2\mathrm{u}_1-1)+ m_1\delta_2+m_2\delta_3\in\mathbb{Z},~~
    \frac{p-q}{2}(2\mathrm{u}_1-1)-m_2\delta_2+m_1\delta_3\in\mathbb{Z}. 
    \end{gathered}\]
which is equivalent to 
\begin{align}
    m_1(\delta_2+(m_1+m_2)\mathrm{u}_1)-m_2(\delta_3+(m_1-m_2)\mathrm{u}_1)\in\mathbb{Z}\label{eq-case21},\\
    m_2(\delta_2+(m_1+m_2)\mathrm{u}_1)+m_1(\delta_3+(m_1-m_2)\mathrm{u}_1)\in\mathbb{Z},\label{eq-case22}\\
    \frac{1}{2}+m_1(\delta_2+(m_1+m_2)\mathrm{u}_1)+m_2(\delta_3+(m_1-m_2)\mathrm{u}_1)\in\mathbb{Z},\label{eq-case23}\\
    -\frac{1}{2}+m_2(\delta_2+(m_1+m_2)\mathrm{u}_1)-m_1(\delta_3+(m_1-m_2)\mathrm{u}_1)\in\mathbb{Z}.\label{eq-case24}
\end{align}
The linear combination of \eqref{eq-case21} and \eqref{eq-case22} with coefficients $m_1$ and $m_2$ yields that $r\big(\delta_2+(m_1+m_2)\mathrm{u}_1\big)\in\mathbb{Z}$. 
Summing equations \eqref{eq-case21} through \eqref{eq-case24} and multiplying $m_1-m_2$ gives    $2p\bigl(\delta_2+(m_1+m_2)\mathrm{u}_1\bigr)\in\mathbb{Z}$. 
Hence by $(r,p)=1$, we derive  %Then, following a similar argument as in the first case, we obtain 
that $\delta_2+(m_1+m_2)\mathrm{u}_1\in\mathbb{Z}$ and therefore,
\[
m_2(\delta_3+(m_1-m_2)\mathrm{u}_1)\in\mathbb{Z},\quad \frac{1}{2}+m_2(\delta_3+(m_1-m_2)\mathrm{u}_1)\in\mathbb{Z}
\]
which can never occur.

This means the immersion $\widetilde{\mathrm{x}}_r(\mathrm{u}_1,\mathrm{u}_2,\mathrm{u}_3)$ %\eqref{eq-example} 
is embedded at 
$$\{(\mathrm{u}_1,\mathrm{u}_2,\mathrm{u}_3)\mid \mathrm{u}_1\in (0,\frac{1}{2})\cup(\frac{1}{2},1), 0\leq \mathrm{u}_2,\mathrm{u}_3<1\}.$$ 
However, it is easy to see that
\[
\widetilde{\mathrm{x}}_r(0,\mathrm{u}_2,\mathrm{u}_3)=x(0,\mathrm{u}_2+\frac{m_1}{r},\mathrm{u}_3-\frac{m_2}{r})=x(0,\mathrm{u}_2+\frac{m_2}{r},\mathrm{u}_3+\frac{m_1}{r}),
\]
\[
\widetilde{\mathrm{x}}_r(\frac{1}{2},\mathrm{u}_2,\mathrm{u}_3)=x(\frac{1}{2},\mathrm{u}_2+\frac{m_1}{r},\mathrm{u}_3+\frac{m_2}{r})=x(\frac{1}{2},\mathrm{u}_2-\frac{m_2}{r},\mathrm{u}_3+\frac{m_1}{r}).
\]
So the image of the minimal immersion $\widetilde{\mathrm{x}}_r: \widetilde{\mathbb{T}}^3_r \rightarrow \mathbb{S}^7$ has self-intersections.

\iffalse{This means distinct $\mathbf{u}\in [0,1)^3$ correspond to distinct $(\theta_1,\cdots,\theta_{12})$.

\[Y=\left(
\begin{array}{cccccccccccc}
 \frac{r+1}{2} & \frac{1-r}{2} & \frac{r+1}{2} & \frac{1-r}{2} & \frac{1+p+q}{2} & \frac{1-p-q}{2} & \frac{1+p-q}{2}& \frac{1-p+q}{2} \\
 u & -u & v & -v & u & -u & -v & v  \\
 -v & v & u & -u & v & -v & u & -u \\
\end{array}
\right).
\]
It is easy to see that for certain solutions $a_i$, $\alpha_i^2+\beta_i^2<a_{2i-1}a{2i}(1\leq i\leq 6)$

{\bf Claim.}
    %It is easy to see that $\{p,q,r\}$ must be one even and two odd. When $r$ is even, 
    Taking $\mathbf{u}=(\frac{1}{2},\frac{1}{2},\frac{1}{2})^t$ in \eqref{eq-em}, we get that all $\langle Y_i,\mathbf{u}\rangle$ are integers. %When $r$ is odd, $\mathbf{u}=(\frac{1}{2},\frac{1}{2},\frac{1}{2})^t$ will do the same. 
    Note that on the torus $\mathbb{R}^3/\Lambda_3$, the point $\mathbf{u}$ is distinct with $(0,0,0)^t$, while their images of $x$ coincide. This means these minimal isometric immersions are not embedded. It is interesting to ask whether embedding implies homogeneity. 
%\end{rem}
}\fi

\section{Isometric deformation and irrationality degree of minimal flat $n$-tori}\label{sec-deform}
In this section, we first address the minimal target dimension problem for minimally immersed flat $n$-tori through isometric deformation analysis. %we first study the minimal target dimension problem of flat $n$-tori admitting minimal immersions into sphere by considering the isometric deformation  of minimal flat $n$-tori. 
Then,  we establish an upper bound for %we conclude this paper by deriving an upper bound for 
the irrationality degrees of minimal flat $n$-tori. 
\subsection{The isometric deformation of minimal flat $n$-tori in spheres}
%Furthermore, we can show the following,
\begin{theorem}\label{thm-deformation}
    Let $x: T^n\triangleq \mathbb{R}^n/\Lambda_n\rightarrow %\mathbb{S}^m
    {\mathbb{S}^{2N-1}}$ be a minimal flat $n$-torus, {where $2N$ is the dimension of the eigenspace of $x$ with respect to the eigenvalue $n$}. Then $x$ can be deformed by a homotopy of minimal isometric immersions into a homogeneous immersion in $\mathbb{S}^{p}$ with $p<n^2+n$.  
\end{theorem}
\begin{proof}
    We denote by $\{Q, Y\}$ the matrix data of $x$, where $Q$ is the Gram matrix of $\Lambda_n$ with respect to a chosen generator, and $Y=\{Y_1, \cdots, Y_N\}$ is the integer vectors  appearing in the minimal immersion $x$. First, the immersion $x$ can deform into a homogeneous immersion into $\mathbb{S}^{2N-1}$ (may not {be} linearly full). Then from Theorem~\ref{thm-vari}, $\frac{Q^{-1}}{n}$ lies in $C_Y$ as the maximal point of the determinant function restricted on $C_Y$, where $C_Y$ is a convex polytope decomposed (not necessarily uniquely) as a union of some simplices of dimension at most $\frac{n(n+1)}{2}{-1}$. Therefore, $\frac{Q^{-1}}{n}$ must lie in the interior of some simplex of dimension at most $\frac{n(n+1)}{2}{-1}$, which means there exists a subset $\tilde{Y}\subset Y$ whose cardinality at most $\frac{n(n+1)}{2}$ such that $\frac{Q^{-1}}{n}\in \overset{\circ}{C_{\tilde{Y}}}$. As a maximum point of $\det$ on $C_Y$, it automatically maximizes $\det$ on  $C_{\tilde{Y}}$. It follows from %the proof of
   Theorem~\ref{thm-vari} that the matrix data $\{Q, \tilde{Y}\}$ provides a homogeneous minimal and isometric immersion of $T^n$ in $\mathbb{S}^{p}$. Such immersion is obviously a special solution of \eqref{eq:flat} and hence this completes the proof.
\end{proof}
Combining Theorem \ref{thm-deformation} with Proposition~\ref{eq-infinite}, we obtain the following corollary. 
{\begin{corollary}
For every rational flat $n$-torus $T^n$, there exist infinitely many positive integers $k$ such that $T^n$ admits a minimal immersion into  $\mathbb{S}^{n^2+n-1}$ realized by the $k$-th eigenfunctions.      
\end{corollary}}
\begin{rem}\
\begin{enumerate}
    \item Combining Theorem~\ref{thm-rational} with Theorem~\ref{thm-deformation} we obtain that up to a dilation,  every flat rational $n$-torus admits a minimal and isometric immersion in $\mathbb{S}^{n^2+n-1}$.  % can be deformed isometrically to a minimal flat torus in $\mathbb{S}^5$. 
    \item
In particular, by Theorem~\ref{thm-deformation}, every minimal flat $2$-torus admits an isometric and minimal immersion in $\mathbb{S}^5$. See below for an illustration by some special examples. {Moreover, for any such flat $2$-torus, there exist infinitely many integers $k$ so that the minimal immersions into $\mathbb{S}^5$ can be realized by the $k$-th eigenfunctions. %This strengthen a conclusion of Bryant in \cite{Bryant}.
}  
\end{enumerate} 
\end{rem}
\begin{example}
By Bryant's theorem, a flat torus $T_{a,b}=\mathbb{R}^2/2\pi(\mathbb{Z}\oplus(a+ib)\mathbb{Z})$ can be immersed into $S^n$ by flat minimal isometric immersions if and only if both $a$ and $b^2$ are rational numbers. 
It is well-known that the eigenfunctions and eigenvalues of $T_{a,b}$  are
\[\cos(pu+\frac{q-pa}{b}v),~~\sin(pu+\frac{q-pa}{b}v), ~~\lambda_{p,q}=p^2+\frac{(q-pa)^2}{b^2}.\]
Consider a special family of $2$-tori with 
\[a=\frac{m}{n}<\frac{1}{2}, ~~b=\sqrt{1-a^2}=\frac{\sqrt{n^2-m^2}}{n}.\] 
For such $2$-tori we have
\[\lambda_{p,q}=\frac{(p^2+q^2)n^2-2pqmn}{n^2-m^2}.\]
One can check that when $\{p,q\}=\{n,0\}$ and $\{p,q\}=\{n,2m\}$, we always have 
\[\lambda_{p,q}=\frac{n^4}{n^2-m^2}\]
with  corresponding eigenfunctions
\[e^{\pm i\frac{n^2v}{\sqrt{n^2-m^2}}}, e^{\pm i\left(nu+\frac{mnv}{\sqrt{n^2-m^2}}\right)}, e^{\pm i\left(nu-\frac{mnv}{\sqrt{n^2-m^2}}\right)}, e^{\pm i\left(2mu+\frac{(n^2-2m^2)v}{\sqrt{n^2-m^2}}\right)}.\]
Then we see that the immersion $f:T_{a,b}\rightarrow S^7\subset \mathbb{R}^8=\mathbb{C}^4$
\[f:=\left(r_1e^{i\frac{n^2v}{\sqrt{n^2-m^2}}}, r_2e^{i\left(nu+\frac{mnv}{\sqrt{n^2-m^2}}\right)}, r_3e^{i\left(nu-\frac{mnv}{\sqrt{n^2-m^2}}\right)}, r_4e^{ i\left(2mu+\frac{(n^2-2m^2)v}{\sqrt{n^2-m^2}}\right)}\right)\]
is a %conformal, flat, 
minimal immersion if and only if 
\[\left\{\begin{aligned}
&    r_1^2+r_2^2   + r_3^2+r_4^2=1,\\
&    n^4r_1^2-n^2(n^2-2m^2)(r_2^2+r_3^2)+(n^4-8m^2n^2+8m^4)r_4^2=0,\\
&  n^2(r_2^2-r_3^2)+2(n^2-2m^2)r_4^2=0.\\
\end{aligned}\right.\]
The second and third equation can be re-written as
\[\left\{\begin{aligned}
&   r_1^2-(b^2-a^2)(r_2^2+r_3^2)+(1-8a^2+8a^4)r_4^2=0,\\
&  r_3^2-r_2^2=2(b^2-a^2)r_4^2.\\
\end{aligned}\right.\]
Setting $r_4=\rho\geq0$, we obtain the solutions to the above equations are
\[\left\{\begin{aligned}
&    r_1^2=\frac{b^2-a^2}{2b^2} -(b^2-3a^2) \rho^2,\\
&    r_2^2=\frac{1}{4b^2}-\rho^2,\\
& r_3^2=\frac{1}{4b^2} +(b^2-3a^2) \rho^2,\\
\end{aligned}\right.\hbox{ with }0\leq\rho\leq \frac{1}{2b}.\]
This provides a family of minimal flat $2$-tori in $S^7$, which is full in $S^7$ when $0<
\rho<\frac{1}{2b}$, and is contained in some $S^5\subset S^7$ when $\rho=0$ or $\frac{1}{2b}$.

Note that when $n=2k$, one can choose $\{p,q\}=\{k,0\}$ and $\{p,q\}=\{k,m\}$ with $\lambda_{p,q}=\frac{k^4}{4k^2-m^2}$, to construct a family of flat minimal tori in $\mathbb{S}^7$. We leave it to interested readers.
\end{example}
\begin{rem}
For a minimal flat $n$-torus $x: T^n\rightarrow \mathbb{S}^m$, if $m>n(n+1)-1$, then the minimal immersion of $T^n$ in $\mathbb{S}^{n(n+1)-1}$, as described in Theorem~\ref{thm-deformation}, is situated on the boundary of the moduli space $\mathcal{M}(T^n)$ of minimal isometric immersions of $T^n$. 
If $\dim{\mathcal{M}(T^n)}>0$, then from the convexity of $\mathcal{M}(T^n)$, any two minimal isometric immersions of $T^n$ can be deformed to each other by a smooth homotopy of minimal isometric immersions. We note that during the deformation of two minimal isometric immersions with equal codimension, the codimension can first increase and later decrease. 
\end{rem}

\subsection{The algebraic irrationality of minimal flat $n$-tori}
%\begin{rem}
%At the end of this paper, we discuss how irrational can a minimal flat $n$-torus be.

Employing methods analogous to those {developed in Subsection~\ref{subsec-3tori}}, %together with the idea in the proof of Theorem \ref{thm-deformation}, 
we derive an upper bound estimate for the algebraic irrationality of minimal flat $n$-tori. The following algebraic lemma will be used. 

\begin{lemma}\label{lem-degree}
    Let $P_1, P_2, \cdots, P_m \in \mathbb{Q}[x_1, x_2,\ldots, x_m]$ be %homogeneous 
    polynomials with respective degrees $d_1, d_2, \cdots, d_m$, %respectively, %in the variables $\{x_1, x_2,\ldots, x_m\}$, with $\deg(P_i)=d_i$ for each $1\leq i\leq m$. 
    and $(\alpha_1, \alpha_2, \cdots, \alpha_m)$ be {an isolated} common zero point of these polynomials. %If  {\color{red}it is isolated%the affine variety defined by the common
     %zeros of $\{P_1, P_2, \cdots, P_m\}$ has dimension $0$
     %}, %$P_1, P_2, \cdots, P_m$ are coprime, 
    Then %for each  $1\leq i\leq m$, 
    the field extension degree $[\mathbb{Q}(\alpha_1, \alpha_2, \cdots, \alpha_m):\mathbb{Q}]$ satisfies %$$[\mathbb{Q}(\alpha_1):\mathbb{Q}]\leq d_1 d_2\cdots d_m,~~~[\mathbb{Q}(\alpha_{i+1}):\mathbb{Q}(\alpha_i)]\leq d_{i+1}\cdots d_m.$$
    %$$[\mathbb{Q}(\alpha_1, \alpha_2, \cdots, \alpha_m):\mathbb{Q}]\leq d_1^m d_2^{m-1}\cdots d_i^{m-i+1} \cdots d_m.$$
    %~~~[\mathbb{Q}(\alpha_{i+1}):\mathbb{Q}(\alpha_i)]\leq d_{i+1}\cdots d_m.
    $$[\mathbb{Q}(\alpha_1, \alpha_2, \cdots, \alpha_m):\mathbb{Q}]\leq d_1 d_2\cdots d_m.$$
\end{lemma}
\begin{proof}
For each $1 \le i \le m$, let $\widetilde{P}_i \in \mathbb{Q}[x_1, x_2, \cdots, x_m, x_{m+1}]$ be the homogenization of the polynomial $P_i$. We denote by $V$ the {projective scheme} defined by $\widetilde{P}_1, \widetilde{P}_2, \cdots, \widetilde{P}_m$ in {$\mathbb{P}^m_{\mathbb{Q}}$}.

{Let $p$ be the corresponding closed point of $(\alpha_1, \alpha_2, \cdots, \alpha_m)$} in $V$. It follows from the {generalized} B\'ezout theorem over the non-{algebraically} closed field  \cite[Proposition 8.4, Page 145]{Fulton} (see also \cite{McKean}) that
\[
l(\mathcal{O}_p(V)) [\mathbb{Q}(\alpha_1, \alpha_2, \cdots, \alpha_m) : \mathbb{Q}] \le d_1 d_2 \cdots d_m,
\]
where $l(\mathcal{O}_p(V))$ denotes the length of the local ring $\mathcal{O}_p(V)$ of $V$ at $p$. {Since $l(\mathcal{O}_p(V)) \ge 1$, the conclusion immediately follows.}
\end{proof}
\begin{theorem}\label{thm-degree}
Suppose $T^n=\mathbb{R}^n/\Lambda_n$ is a flat $n$-torus that admits a minimal isometric immersion into some sphere, equipped with the corresponding matrix data $\{Q, Y\}$. Let $K/\mathbb{Q}$ be the minimal field extension containing all entries of $Q$. Then %the Gram matrix $Q$ of $\Lambda_n$, and $Y=\{Y_1, Y_2, \cdots, Y_N\}$ be the integer set associated to this minimal immersion. 
%We assume that %$\{Y_1,\cdots,Y_k\}\subset \mathbb{Z}^n$ satisfy 
%\[
%\mathrm{rank}\{Y_1,\cdots,Y_n\}=n,\quad 
%\mathrm{rank}\{Y_iY_i^t\}_{1\leq i\leq k}=k.
%\rk\{Y_jY_j^t\,|\, Y_j\in Y\}=k. 
%\]
%Employing methods analogous to those {developed in Subsection~\ref{subsec-3tori}}, we can derive the following corse upper bound for 
the extension degree satisfies  
$$[K:\mathbb{Q}]\leq \min\{%(2k-n-1)^{\frac{(n-2)(n-1)}{2}},
(n-1)^{k-1}, (n-1)^s\},$$ %\frac{s(s+1)}{2}
where $k=\rk\{Y_jY_j^t\,|\, Y_j\in Y\},~s=\frac{n(n+1)}{2}-k$.
\end{theorem}
%Let $\{Q_1,\cdots,Q_s\}\subset M(\mathbb{Q})\,(s=\frac{n(n+1)}{2}-k)$ be a basis for the solution space of $\langle X,Y_iY_i^t\rangle=0$, $Q_0=\sum_{i=1}^kc_iY_iY_i^t$ where $\{c_i\}$ is the only solution to the linear system $\langle\sum_{j=1}^kx_jY_jY_j^t,Y_iY_i^t\rangle=1$. Obviously, $Q_0$ is also rational. 
\begin{proof}
Without loss of generality, we assume that 
$$\rk\{Y_1Y_1^t, Y_2Y_2^t, \cdots, Y_kY_k^t\}=k.$$
%In fact, on the one hand, 
%It follows 

On the one hand, from the proof of Theorem~\ref{thm-deformation}, we may assume 
$\frac{Q^{-1}}{n}$ lies in the convex hull of 
$\{Y_1Y_1^t, Y_2Y_2^t, \cdots, Y_kY_k^t\}$, which may be  locally parameterized as 
$$Y_kY_k^t+\lambda_1(Y_1Y_1^t-Y_kY_k^t)+\lambda_2(Y_2Y_2^t-Y_kY_k^t)+\cdots+\lambda_{k-1}(Y_{k-1}Y_{k-1}^t-Y_kY_k^t).$$
Then it follows from Theorem~\ref{thm-vari} that, as the unique
critical point of the determinant function restricted on this convex hull, the parameter of $\frac{Q^{-1}}{n}$ satisfies a system of $k-1$ polynomial equations in the variables $(\lambda_{1},\cdots,\lambda_{k-1})$, with %where each equation has 
rational coefficients %in $\mathbb{Q}$ 
and degree $n-1$. Moreover, it is an  isolated common zero point. Therefore, by Lemma~\ref{lem-degree} we have $$[K:\mathbb{Q}]\leq (n-1)^{k-1}.$$   

On the other hand, {similar to} the approach introduced at the beginning of Section~\ref{sec-3tori}, we can parameterize $W_Y$ as follows, %  So the solutions for system $\langle X,Y_iY_i^t\rangle=1$ can be expressed by
\[
W_Y=\{Q_0+t_1Q_1+\cdots+t_sQ_s|t_i\in \mathbb{R}\}\cap \Sigma_+, 
\]
where $Q_0%=\sum_{i=1}^kc_iY_iY_i^t
\in {\mathrm{Sym_n}}\cap GL(n,\mathbb{Q})$ is the unique matrix in $\mathrm{Span}\{Y_jY_j^t\,|\, Y_j\in Y\}$ satisfying $$\langle Q_0,Y_iY_i^t\rangle=1,~~~1\leq i\leq \sharp(Y),$$ 
and $\{Q_1,\cdots,Q_s\}\subset {\mathrm{Sym_n}}\cap GL(n,\mathbb{Q})$ is a basis of the orthogonal complement of $\mathrm{Span}\{Y_jY_j^t\,|\, Y_j\in Y\}$. 
%If $\mathcal{N}\cap\Sigma_+=\emptyset$ then we do not get any torus. So we assume $\mathcal{N}\cap\Sigma_+\not=\emptyset$, then there exists a unique extreme value for the determinant function defined in $\mathcal{N}\cap\Sigma_+$. 
%Note that the determinant function restricted on $W_Y$ can be expressed as %For any $P\in W_Y$, $\det(P)$ is 
%a rational polynomial of $(t_1,\cdots,t_s)$ with degree $n$. %order $n$ defined in $(t_1,\cdots,t_s)\in\mathbb{R}^s$. 
Consequently, as the critical point of the determinant function restricted on $W_Y$, the parameter of $Q$ satisfies a system of $s$ polynomial equations in the variables  $(t_1,\cdots,t_s)$, with %where each equation has 
rational coefficients %in $\mathbb{Q}$ 
and degree $n-1$. So by Lemma~\ref{lem-degree}, we also have 
%$$[K:\mathbb{Q}]\leq (n-1)^{\frac{s(s+1)}{2}}.$$ 
$$[K:\mathbb{Q}]\leq (n-1)^{s}.$$ 
\end{proof}
%So the extreme value of $det(P)$ must take place at the solutions of
%\[\frac{\partial det(P)}{\partial t_i}=0,\quad 1\leq i\leq s,\]
%which are $s$ polynomial equations of order $(n-1)$ in $\mathbb{Q}$. As a consequence, the solution $\{t_i\}$ must take values in a $s(n-1)$-extension field of $\mathbb{Q}$. 
%\end{rem}
%\begin{rem}

\begin{proof}[Proof of  Theorem~\ref{thm-2}]
       Note that $$\max_{n\leq k\leq \frac{n(n+1)}{2}}\min\{k-1, \frac{n(n+1)}{2}-k\}={[\frac{(n-1)(n+2)}{4}]}.$$
    It follows that the extension degree in Theorem~\ref{thm-degree} satisfies 
    \begin{equation}
   \label{eq-upper} [K:\mathbb{Q}]\leq (n-1)^{\lfloor\frac{(n-1)(n+2)}{4}\rfloor}.
    \end{equation}
\end{proof} 
 
  \begin{rem} {For the case of  $3$-tori, the upper bound given in \eqref{eq-upper} equals $4$, and it  is optimal by Theorem~\ref{thm-3irra} and examples in Section 4.    
 %According to Theorem~\ref{thm-3irra}, the upper bound established in Theorem~\ref{thm-degree} is optimal for $3$-tori, realized by $(n-1)^{s}$. 
  For the case of  $n$-tori with $n\geq4$,  it is unknown whether the corresponding upper bound is sharp. %It seems that the highest extension degree may be less than $(n-1)^{\frac{(n-1)(n+2)}{4}}$, realized around $k=\frac{n^2+n+2}{4}$.
  }
\end{rem}

\textbf{Acknowledgement:}
The first author and the third author is supported by NSFC No. 12171473. The second author is supported by NSFC No. 12371052. The third author is also partially supported by the Fundamental Research Funds for Central Universities. %The authors are grateful to Prof. C.P. Wang for his continuous encouragement on this work. 
%The authors are thankful to Prof. Q.-S. Chi and Prof. R. Kusner for valuable discussions. We appreciate the
%referee for many valuable suggestions. 
%PW was partly supported by the Project 11971107 of NSFC. ZXX is supported by Yueqi Scholars of CUMTB. 
\vskip 0.3cm
%\textbf{Conflict of interest:}
%On behalf of all authors, the corresponding author states that there is no conflict of
%interest. 

%\bibliographystyle{plain}
%\bibliography{ref}

\iffalse
\vspace{5mm} \noindent Ying L\"u\\
{\small\it  School of Mathematical Sciences, Xiamen University, Xiamen, 361005, P. R. China.\\
Email: {lueying@xmu.edu.cn}}

\vspace{5mm} \noindent Peng Wang\\
{\small\it  School of Mathematics and Statistics, FJKLMAA, Fujian Normal University, Fuzhou 350117, P. R. China.\\
Email: {pengwang@fjnu.edu.cn}}

\vspace{5mm} \noindent Zhenxiao Xie\\
{\small\it Department of Mathematics, China University of Mining and Technology (Beijing),
Beijing 100083, P. R. China.
Email: {xiezhenxiao@cumtb.edu.cn}}

\fi
\end{document}